\newif\ifdraft

\drafttrue

\ifdraft

\documentclass[11pt, reqno]{amsart}
\usepackage{lmodern}

\usepackage[a4paper, margin=1in]{geometry}

\else

\documentclass[reqno]{amsart}
\usepackage{lmodern}

\usepackage[a4paper, margin=.75in]{geometry}
\fi

\usepackage{amsmath, amsthm, thmtools, amsfonts, amssymb, mathtools}
\usepackage{pdflscape, blkarray, multirow, booktabs}
\usepackage{mathrsfs}

\usepackage{amstext} 
\usepackage{array}   
\newcolumntype{L}{>{$}l<{$}} 

\usepackage[dvipsnames]{xcolor}
\usepackage{hyperref}
\hypersetup{
	colorlinks = true,
	linkcolor = {Blue},
	citecolor={BrickRed},
}

\usepackage{enumerate}
\usepackage{makecell}
\usepackage{graphicx}
\usepackage{float}

\newcommand{\defeq}{\vcentcolon=}

\newcommand{\tk}[1][4]{T_{#1}(K)}
\newcommand{\avec}[1]{\vec{\mathbf{a}}_{\{#1\}}}
\newcommand{\avect}[1]{\vec{\mathbf{a}}_{#1}}

\newcommand{\lambdak}{\mu_{i_1\ldots i_k}}
\newcommand{\mip}{multi-index $p$-inductive polynomial}

\theoremstyle{definition}
\newtheorem{con}{Conjecture}

\newtheorem{ques}{Question}
\newtheorem{thm}{Theorem}[section]
\newtheorem*{thmIntro}{Theorem}

\newtheorem{defn}[thm]{Definition}
\newtheorem{lemma}[thm]{Lemma}
\newtheorem{example}[thm]{Example}

\newtheorem{proposition}[thm]{Proposition}

\newtheorem*{conjecture}{Conjecture}

\newcommand{\ord}[1][p]{\mathrm{ord}(#1)}
\newcommand{\diag}{\mathrm{diag}}
\newcommand{\D}{\ensuremath{\mathcal{D}}}
\newcommand{\ch}{\textbf{char}}

\newcommand{\nocontentsline}[3]{}
\newcommand{\tocless}[2]{\bgroup\let\addcontentsline=\nocontentsline#1{#2}\egroup}
\definecolor{newBlue}{HTML}{8611e8}
\newcommand{\hf}{\hspace{0.2cm}}

\newcommand\scalemath[2]{\scalebox{#1}{\mbox{\ensuremath{\displaystyle #2}}}}

\begin{document}
\allowdisplaybreaks
\title[The image of polynomials and Waring type problems]{The image of polynomials and Waring type problems on upper triangular matrix algebras}

\author[S. Panja]{Saikat Panja}
\address{Department of Mathematics, IISER Pune \\ Maharashtra, India}
\email{panjasaikat300@gmail.com}
\author[S. Prasad]{Sachchidanand Prasad}
\address{Department of Mathematics and Statistics, IISER Kolkata \\ West Bengal, India}
\email{sp17rs038@iiserkol.ac.in}



\begin{abstract}
    Let $p$ be a polynomial in non-commutative variables $x_1,x_2,\ldots,x_n$ with constant term zero over an algebraically closed field $K$. The object of study in this paper is the image of this kind of polynomial over the algebra of upper triangular matrices $T_m(K)$. We introduce a family of polynomials called multi-index $p$-inductive polynomials for a given polynomial $p$. Using this family we will show that, if $p$ is a polynomial identity of $T_t(K)$ but not of $T_{t+1}(K)$, then $p \left(T_m(K)\right)\subseteq T_m(K)^{(t-1)}$. Equality is achieved in the case $t=1,~m-1$ and an example has been provided to show that equality does not hold in general. We further prove existence of $d$ such that each element of $T_m(K)^{(t-1)}$ can be written as sum of $d$ many elements of $p\left( T_m(K) \right)$. It has also been shown that the image of $T_m(K)^\times$ under a word map is Zariski dense in $T_m(K)^\times$.
\end{abstract}

\date{\today}
\maketitle

\setcounter{tocdepth}{3}

\frenchspacing 

\tableofcontents
\section{Introduction} \label{sec:introduction}
There has been growing interest in polynomial equation (and word equation) in associative algebras (and groups), in recent years.
This line of research has an origin in the work of two great mathematicians of the last century, namely Armand Borel (for group-theoretic questions) and Irving Kaplansky (for ring theoretic questions). 
We should keep in mind that the results for groups differ substantially from those concerning algebras, mainly because of the existence of non-trivial identities (polynomial identities, more precisely) for simple matrix algebras, which can not occur for simple matrix groups due to Tit's alternative. 
There is a plethora of tools used to solve these problems in the case of algebras, mostly coming from algebraic geometry, and arithmetic properties of polynomial equations in associative algebras,
whereas people have been using algebraic groups, character theory, and deep results 
from the classification of finite group theory to handle these problems in case of groups. Let us start with a brief introduction to both of the problems in the subsequent discussion.

Let $F_n$ denotes the free group on $n$ symbols $x_1,x_2,\ldots,x_n$ and $K$ be an algebraically closed field.
The elements of the group ring $K[F_n]$ will be called as
\textit{polynomials in non-commutative variables}. Given an algebra $\mathcal{A}$ over $K$ and $p\in K[F_n]$, consider the following map
\begin{align*}
    \Phi_p:\mathcal{A}^n\longrightarrow \mathcal{A},\quad (a_1,a_2,\ldots,a_n)\mapsto p(a_1,a_2,\ldots,a_n).
\end{align*}
These are known as \textit{polynomial maps}  on $\mathcal{A}$. Obviously, we need the polynomial to be in non-commutative variables as the algebra need not be commutative. Images of polynomials are crucial to answer different questions of non-commutative algebra. A problem due to Kaplansky \cite{Ka57} was settled down after the construction of central polynomials 
(i.e., a polynomial $f$ such that $f(\mathcal{A})$ lies inside $Z(\mathcal{A})$) by Formanek \cite{Fo72}. There are questions regarding the image of polynomial maps on an associative algebra, for example, $M_n(K)$, motivated by several questions raised by Kaplansky. 
Note that the image can vary substantially depending upon the polynomial. For example, it can be a field of scalars when the polynomial is multilinear central (e.g. multilinearization of the polynomial of $[x_1,x_2]^2$ is a central polynomial of the smallest possible degree for $M_2(K)$), 
or take the case $f(x_1,x_2)=x_1x_2-x_2x_1$, where the image is $[M_n(K),M_n(K)]$, the set of all trace $0$ matrices (this is due to a theorem of Albert and Muckenhoupt \cite{AlMu57}). 
Determining the image is difficult, as the image need not be a vector space (at least not in an obvious way), although it is a cone invariant under matrix conjugation. At this point let us digress a little, to give a glimpse of the growth of the problem in the context of group theory.

The parallel study in group theory is concerned with the maps corresponding to an element of $F_n$, known as \textit{word maps}. This problem is motivated by the
classical Waring problem in number theory
which asks whether every positive integer can be written as $g(k)$ many $k$-th powers for a suitable function $g$. For a prime $q$, to investigate 
the image of a word map in pro-$q$-groups, it is important to know the image of Lie polynomials over the algebra of $n\times n$ matrices. Indeed, this has helped Zelmanov \cite{Ze05} to prove that for 
$q\gg n$, the free pro-$q$-group can not be embedded into the algebra of $n\times n$ matrices. It has been proved that for semisimple algebraic group $G$ and any word $w\in F_n$,
the map $w:G^n\longrightarrow G$ is dominant, by Borel \cite{Bo83} and Larsen \cite{La04} independently. This has further motivated mathematicians to investigate word maps on finite simple groups of Lie type.
 A recent breakthrough in this direction is the affirmation of
 Ore's conjecture (which states that the commutator map corresponding to the word $xyx^{-1}y^{-1}\in F_2$ is surjective in the case of finite non-abelian simple groups), by Liebeck, O'Brien,
 Shalev and Tiep \cite{LiObSh10}, using the methods from character theory. 
 Indeed the authors prove that if $G$ is any quasisimple classical group over a finite field, 
 then every element of $G$ is a commutator, using character-theoretic results due to 
 Frobenius. Later they proved results about the product of squares in finite non-abelian 
 simple group in \cite{LiObSh12}.
 They proved that 
 every element of a non-abelian finite simple group $G$ is a product of two squares. For a survey of these problems in the context of group theory, we refer the reader to the excellent article \cite{Sh09} due to Shalev and a survey article \cite{LaShTi11} due to Larsen, Shalev and Tiep.

Coming back to the case of associative algebras, let us start by stating the famous  Kaplansky-Lvov conjecture which is concerned with the image of multilinear maps in non-commutative variables over the algebra $M_n(K)$ of $n\times n$ matrices. It states the following;
 \begin{con}
 Let $p$ be a multilinear polynomial in $m$ variables, over a field $k$. Consider the image set of $M_n(k)$ under the
 polynomial map. Then $\textup{Im}\,\Phi_p$ is either $\{0\}$, $\{aI:a\in k\}$, $\mathfrak{sl}_n(k)$ or $M_n(k)$, where $\mathfrak{sl}_n(k)$ denotes the set of all
 trace $0$ matrices in $M_n(k)$.
 \end{con}
 At the time of writing this paper the conjecture stands open for $n\geq 2$. The case $n=2$ has been solved for quadratically closed fields in \cite{KaMaRo12} by Kanel-Belov, Malev, and Rowen and for $n=3$  substantial progress has been made in \cite{KaMaRo16} by the same authors. One of the lemma in \cite{KaMaRo12} (essentially a corollary of a theorem of Amitsur) is that a generic matrix algebra with characteristic coefficients adjoined 
 does not contain zero divisors. This enables to study of eigenvalues of elements of the image of the polynomial. But, unfortunately, this method of comparing eigenvalues starts to break down for larger $n$. This necessitates the use of more intrinsic machinery of algebraic geometry for the case of $3\times3$ matrices.  
Some partial results are known due to \cite{KaMaRo16_survey}, in case $n=4$. For survey of these results in case of algebras, we suggest the reader the beautiful article \cite{KaMaRoYa20} by Kanel-Belov, Malev, Rowen, and Yavich. Since over an algebraically closed field, any matrix can be conjugated to an upper triangular matrix, variations of the Kaplansky-Lvov conjecture are being studied for upper triangular matrix algebras $\tk[m]$,  in recent times. The following conjecture was posed in 2019 by Fagundes and de Mello \cite{FaMe19}.
 \begin{con}
 The images of a multilinear polynomial in non-commutative variables over a field $K$, on the algebra $T_m(K)$  is a vector space.
 \end{con}
 \noindent This conjecture has been settled recently by Luo, Wang in \cite{LuWa22} and Gargate, de Mello in \cite{GaMe_unpub} independently. The results are as follows.

 \begin{thmIntro}[Luo, Wang]
     Let $m\ge 1$ and $n\ge 2$  be two integers. Let $K$  be an infinite field and $p \left(x_1,\ldots,x_m\right)$ be a nonzero multilinear polynomial in non-commutative variables over $K$. Suppose that $|K|> \frac{n(n-1)}{2}$. We have that $p \left(T_n(K)\right)=T_n(K)^{(k)}$ for some integer $-1\le k\leq \left[\frac{m}{2}\right]-1$, where $\left[\frac{m}{2}\right]$ is the integer part of $\frac{m}{2}$.
 \end{thmIntro}

 \begin{thmIntro}[Gargate, de Mello]
     Let $K$ be an infinite field and $f\in K \left[F_n\right]$ be a multilinear polynomial. Then the image of $f$  on $T_m(K)$ is $J^r$ if and only if $f$ has commutator degree $r$ where $J$ is the Jacobson radical of $T_m(K)$.
 \end{thmIntro}
 
 In this paper, we consider the problem of finding the images of non-commutative polynomials evaluated on upper triangular matrix algebras.
\begin{ques}\label{ques:Question}
     Let $p$ be a polynomial in non-commutative variables, satisfying $p(\mathbf{0})=0$, over an algebraically closed field $K$. Consider the algebra of upper triangular matrices over $K$, to be denoted by $T_m(K)$. Then what are the possible images $p(T_m(K))$?
\end{ques}
\noindent This question has been answered completely in case of $n=2,~ \ord \ge 1$ (see \autoref{defn:order} for the definition of $\ord$), by Chen, Luo and Wang \cite{WaZhLu21} and $n=3,~ \ord \ge 1$ \cite{WaZhLu22}. Furthermore, the properties of the image have been proved for $\ord = 0$.
The results are as follows.
\begin{thmIntro}[Theorem 3.1, \cite{WaZhLu21}]
Let $p(x_1,\ldots, x_n)$ be a polynomial with zero constant term in non-commutative
variables over an algebraically closed field $K$. Then one of the following
statements holds:
\begin{enumerate}
    \item If $\ord\geq 2$ then $p(T_2(K))=0$;
    \item If $\ord=1$ the $p(T_2(K))=T_2(K)^{(0)}$;
    \item If $\ord=0$ then 
    \begin{align*}
        p(T_2(K))=T_2(K)\setminus\left\{\begin{pmatrix}c&a\\&c\end{pmatrix}:c\in K\setminus S(K),a\in K^\times\right\},    
    \end{align*}
    where $S(K)$ denotes the set of strange points (for definition check \cite[Definition 2.1]{WaZhLu21}) of $K$,
    is a Zariski dense subset of $T_2(K)$. Suppose further that
$\ch(K)=0$  or $\ch(K)>d$, where $d$ is the degree of $p$. We have that $|K\setminus S(K)|<d$.
\end{enumerate}
\end{thmIntro}
\begin{thmIntro}[Theorem 3.1, \cite{WaZhLu22}]
Let $p(x_1,\ldots, x_n)$ be a polynomial with zero constant term in non-commutative
variables over an algebraically closed field $K$. Then one of the following
statements holds:
\begin{enumerate}
    \item If $\ord\geq 3$ then $p(T_3(K))=0$;
    \item If $\ord=2$ the $p(T_3(K))=T_3(K)^{(1)}$;
    \item If $\ord=1$ the $p(T_3(K))=T_3(K)^{(0)}$;
    \item If $\ord=0$ then $p(T_3(K))$ is a Zariski dense subset of $T_3(K)$.
\end{enumerate}
\end{thmIntro}

    In this paper, we have studied \autoref{ques:Question} for general $m$, in \autoref{lem:ZariskiDense}, \autoref{thm:order_m-1}, \autoref{thm:order_1}, and \autoref{thm:order_t}. We take a unified approach to the general case. Our results are in terms of polynomial identities of the matrix algebra. These polynomials are known and are a generalization of Lie polynomials of a certain kind. We will be using concepts from basic algebraic geometry, a special family of polynomials to be called \mip s, and some polynomial manipulation to conclude our result.

\noindent\textbf{Organization of the paper:}
We start with a few definitions (e.g. order of a polynomial) and recall necessary theorems (e.g. polynomial identities of the upper triangular matrix algebras) in \autoref{sec:prelim}. Some examples have been provided in \autoref{sec:prelim} for the ease of understanding. After that, we present a proof for the case $n=4$ in \autoref{sec:n=4} concerning the images of non-commutative polynomials for the upper triangular matrix algebra of size $4$, which has been further generalized in \autoref{sec:mainThm}.
Note that this section is motivated by the results of \cite{WaZhLu21} and \cite{WaZhLu22}. Therein the authors used polynomials $p_{i}$ and $p_{ij}$ (associated with a given polynomial $p$) to conclude their results. We observe that this family gets bigger for the case $n=4$, as it includes another family of polynomials namely, $p_{ijk}$. This
necessitates the introduction of \mip s. 
In \autoref{sec:multiIndex}, firstly we present an algorithm for computing particular entries of matrices which are images under a polynomial map. 
This enables us to introduce the notion of \mip s for a given polynomial $p\in K[F_n]$. These are crucial in our study of images of polynomials, as can be seen from \autoref{sec:multiIndex}. This is coupled with an example of a diagrammatic presentation of multiplication, which we believe will help to understand the pathway of generalization. The \mip s corresponding to a polynomial $p$ has further been used to draw our main results in \autoref{sec:mainThm}. The main results in \autoref{sec:mainThm} show that in some cases the images are actually vector space. We provide an example (see \autoref{eg:counterExample}) here to show that this is not true in general. In \autoref{sec:Waring} we discuss that Waring type problem with respect to the polynomial $p$ and its image. The results of \autoref{sec:mainThm} and \autoref{sec:Waring} are summarized in \autoref{thm:mainThm}, which can be treated as the main theorem of this paper. Although we will assume the field to be algebraically closed while stating most of the results, quite some of them are true in the case of infinite fields. We note that \autoref{lem:ZariskiDense} can be altered to conclude about image of $\tk[m]^\times$ under a word map. This has been presented in \autoref{sec:groupMap} as \autoref{thm:groupMap}.

\noindent While communicating with the journal we came to know a similar work has been done in \cite{ChWa22}.

\section{Preliminaries}\label{sec:prelim}
Throughout this section $T_m(K)$ will denote the algebra of upper triangular matrices of size $m$ over the field $K$. For a ring $R$, the set of all invertible elements of $R$ will be denoted by $R^\times$. Consider $T_m(K)$ as a subspace of $T_{m+1}(K)$ by the embedding
\begin{align*}
    T_{m+1}(K)=\left\{\begin{pmatrix}A&\vec{v}\\0&a\end{pmatrix}:A\in T_m(K),\vec{v}\in K^m,a\in K\right\}.
\end{align*}
Then we have a chain of algebras as follows
\begin{align*}
    K=T_{1}(K)\subseteq T_2(K)\subseteq T_3(K)\cdots.
\end{align*}
We reserve the notation $T_m(K)^{(t)}$ to denote all the matrices in $T_m(K)$ such that $ij^{\text{th}}$ entry is $0$ whenever $j-i\leq t$.
\begin{example}
Take $m=4$ and $t=2$. Then
\begin{align*}
    T_4(K)^{(2)}=\left\{\begin{pmatrix}0&0&0&x\\
    &0&0&0\\&&0&0\\&&&0\end{pmatrix}:x\in K\right\}.
\end{align*}
\end{example}
Given a polynomial $f\in K[F_n]$ and an associative algebra $\mathcal{A}$, $f$ is called a \textit{polynomial identity of $\mathcal{A}$} if 
\begin{align*}
    f(a_1,a_2,\ldots, a_n)=0,\quad\text{for all }a_1,a_2,\ldots,a_n\in \mathcal{A}.
\end{align*}
For an algebra $\mathcal{A}$, we denote the set of all polynomial identities as $\mathrm{Id}(\mathcal{A})$. We have the following sequence of subsets
\begin{align*}
    \mathrm{Id}(K)\supseteq\mathrm{Id}(T_2(K))\supseteq\mathrm{Id}(T_3(K))\supseteq\cdots.
\end{align*}
\begin{example}
Consider the polynomial $p\in K[F_2]$ as $p(x_1,x_2)=(x_1x_2-x_2x_1)^2$. Take $u_i=\begin{pmatrix}a_i&c_i\\0&b_i\end{pmatrix}\in T_2(K)$ for $i=1,2$. Then we have
\begin{align*}
    &\left[\begin{pmatrix}a_1&c_1\\&b_1\end{pmatrix}\begin{pmatrix}a_2&c_2\\&b_2\end{pmatrix}-\begin{pmatrix}a_2&c_2\\&b_2\end{pmatrix}\begin{pmatrix}a_1&c_1\\&b_1\end{pmatrix}\right]^2\\
    =&\begin{pmatrix}0&a_1c_2+b_2c_1-a_2c_1-b_1c_2\\&0\end{pmatrix}^2\\
    =&\begin{pmatrix}0&0\\&0\end{pmatrix}.
\end{align*}
But note that 
\begin{align*}
    \left[\begin{pmatrix}
    1&2&0\\
    &1&3\\
    &&1
    \end{pmatrix}\begin{pmatrix}
    1&4&0\\
    &2&5\\
    &&1
    \end{pmatrix}-\begin{pmatrix}
    1&4&0\\
    &2&5\\
    &&1
    \end{pmatrix}\begin{pmatrix}
    1&2&0\\
    &1&3\\
    &&1
    \end{pmatrix}\right]^2\neq\begin{pmatrix}0&0&0\\&0&0\\&&0\end{pmatrix}.
\end{align*}
Thus, $p\in\mathrm{Id}(T_2(K))$ but $p\notin \mathrm{Id}(T_3(K))$. This motivates us to define the following.
\end{example}
\begin{defn}\label{defn:order}
A polynomial $p\in K[F_n]$ will be called to have \emph{order} $t$ if 
$p\in\mathrm{Id}(T_t(K))$ but $p\not\in\mathrm{Id}(T_{t+1}(K))$ where $t\geq 1$, and will be denoted by $\ord$. A polynomial $p$ will be said to have order $0$, if $p(K)\neq 0$.
\end{defn}

We would like to point out that this coincides with the definition of commutator degree of a polynomial \cite[Definition 3.1]{GaMe_unpub}.

\begin{example}\label{eg:order2}
The polynomial $p(x_1,x_2)=\left(x_1x_2-x_2x_1\right)^2$ is of order $2$.
\end{example}
The polynomial identities for the algebra of upper triangular matrix are well known, as stated in the following theorem.
\begin{thm}\cite[Theorem 5.4]{Dr2000}
Let $K$ be an infinite field and let $T_m(K)$ be the algebra of
 upper triangular matrices of size $m$.
The polynomial identity
\begin{align*}
[x_1 , x_2 ][x_3,x_4]\cdots[x_{2 m-1} , x_{2 m} ] = 0
\end{align*}
forms a basis of the polynomial identities of $T_m(K)$, where $[x_l,x_{l+1}]=x_1x_{l+1}-x_{l+1}x_l$.
\end{thm}
Next we recall some terminology from algebraic geometry, for completeness. The space $K^n$ will be referred to as 
\emph{affine $n$-space}. A \emph{closed set} is given by common zeros of ideal in $K[y_1,y_2,\ldots,y_n]$, i.e., $T\subseteq K^n$ is closed if there exists an ideal $I\subseteq K[y_1,y_2,\ldots,y_n]$ such that
\begin{align*}
    T=\{(z_1,z_2,\ldots,z_n)\in K^n:f(z_1,z_2,\ldots,z_n)=0\quad\text{for all }f\in I\}.
\end{align*}
The set of complements of these closed sets form the so called \emph{Zariski topology}. 
A subset $U\subseteq K^n$ will be called \emph{Zariski dense}, if for all open sets $V\subseteq K^n$, we have $U\cap V\neq\emptyset$.
This is the usual definition of dense subsets in topological sense. A subset of the affine space will always be considered to be endowed with Zariski subspace topology. We end this section by providing the structure of the general polynomial we will be working with. Henceforth, $p\in K[F_n]$ will be assumed to be 
satisfying $p(\mathbf{0})=0$, and 
\begin{equation}\label{eq:generalPolynomial}
    p(x_1,x_2,\ldots,x_n)=\displaystyle\sum\limits_{k=1}^d\left(\sum\limits_{1\leq i_1,i_2,\ldots,i_k\leq n}\lambdak x_{i_1}x_{i_2}\ldots x_{i_k}\right),
\end{equation}
where $d$ is the degree of the polynomial and $\lambdak\in K$.

\section{Image of polynomials on \texorpdfstring{$\tk[4]$}{T4K}}\label{sec:n=4}
In this section we will prove results about the image of polynomials on $\tk$ and then in \autoref{sec:mainThm} using this we will generalize our idea. We start with the following lemma which describes the entries of the product of matrices.

\begin{lemma}\label{lem:product}
    Let $k\ge 2$ be an integer. Let 
    \begin{displaymath}
        u_i = 
        \begin{pmatrix}
            a_{11}(i) & a_{12}(i) & a_{13}(i) & a_{14}(i) \\
            & a_{22}(i) & a_{23}(i) & a_{24}(i) \\
            &           & a_{33}(i) & a_{34}(i) \\
            &           &           & a_{44}(i)
        \end{pmatrix} \in \tk,
    \end{displaymath}
    for all $i = 1,2,\ldots,n$. For any $1\le i_1,\ldots,i_k\le n$, we have
    \begin{displaymath}
        \Lambda_k = \prod_{j=1}^k u_{i_j} = 
        \begin{pmatrix}
            \displaystyle\prod_{j=1}^k a_{11}(i_j) & \lambda_{12} & \lambda_{13} & \lambda_{14} 
            \\[1ex]
            & \displaystyle\prod_{j=1}^k a_{22}(i_j) & \lambda_{23} & \lambda_{24} 
            \\[1ex]
            & & \displaystyle\prod_{j=1}^k a_{33}(i_j) & \lambda_{34} 
            \\[1ex]
            & & & \displaystyle\prod_{j=1}^k a_{44}(i_j)
        \end{pmatrix},
    \end{displaymath}
    where 
    \begin{align*}
        \lambda_{ll+1} & = \sum_{s=1}^{k}\left(\prod_{j=1}^{s-1}a_{ll}(i_j)\right) a_{l l+1}(i_s) \left(\prod_{j=s+1}^{k}a_{l+1l+1}(i_j)\right)
        \\[2ex]
        \lambda_{ll+2} & = \sum_{s=1}^{k}\left(\prod_{j=1}^{s-1}a_{ll}(i_j)\right) a_{l l+2}(i_s) \left(\prod_{j=s+1}^{k}a_{l+2l+2}(i_j)\right) 
        \\[1ex]
        & +  \sum_{1\le s<t\le k}\left(\prod_{j=1}^{s-1}a_{ll}(i_j)\right) a_{l l+1}(i_s) \left(\prod_{j=s+1}^{t-1}a_{l+1l+1}(i_j)\right)a_{l+1 l+2}(i_t)  \left(\prod_{j=t+1}^{k}a_{l+2l+2}(i_j)\right)
        \\[2ex]
        \lambda_{14} & = \sum_{s=1}^{k}\left(\prod_{j=1}^{s-1}a_{11}(i_j)\right) a_{14}(i_s) \left(\prod_{j=s+1}^{k}a_{44}(i_j)\right) 
        \\[1ex]
        & +  \sum_{1\le s<t\le k}\left(\prod_{j=1}^{s-1}a_{11}(i_j)\right) a_{12}(i_s) \left(\prod_{j=s+1}^{t-1}a_{22}(i_j)\right)a_{24}(i_t)  \left(\prod_{j=t+1}^{k}a_{44}(i_j)\right)
        \\[1ex]
        & +  \sum_{1\le s<t\le k}\left(\prod_{j=1}^{s-1}a_{11}(i_j)\right) a_{13}(i_s) \left(\prod_{j=s+1}^{t-1}a_{33}(i_j)\right)a_{34}(i_t)  \left(\prod_{j=t+1}^{k}a_{44}(i_j)\right)
        \\[1ex]
        & +  \sum_{1\le s<t<u\le k}\left[\left(\prod_{j=1}^{s-1}a_{11}(i_j)\right) a_{12}(i_s) \left(\prod_{j=s+1}^{t-1}a_{22}(i_j)\right)a_{23}(i_t)  \left(\prod_{j=t+1}^{u-1}a_{33}(i_j)\right)\right.
        \\[1ex]
        & \kern 2cm \left.a_{34}(i_u)  \left(\prod_{j=u+1}^{k}a_{44}(i_j)\right)\right]
    \end{align*}
\end{lemma}

\begin{proof}
    We will prove this by the method of induction on $k$. Let $k=2$. Then we have
    \begin{align*}
        u_{i_1}u_{i_2} & = 
        \begin{pmatrix}
            a_{11}(i_1)a_{11}(i_2) 
            & a_{11}(i_1)a_{12}(i_2) 
            & a_{11}(i_1)a_{13}(i_2) 
            & a_{11}(i_1)a_{14}(i_2)
            \\
            & +a_{12}(i_1)a_{22}(i_2)
            & +a_{12}(i_1)a_{23}(i_2)
            & +a_{12}(i_1)a_{24}(i_2)
            \\
            & 
            & +a_{13}(i_1)a_{33}(i_2)
            & +a_{13}(i_1)a_{34}(i_2)
            \\
            & 
            & 
            & +a_{14}{(i_1)a_{44}(i_2)}
            \\[2ex] 
            & a_{22}(i_1)a_{22}(i_2)
            & a_{22}(i_1)a_{23}(i_2)
            & a_{22}(i_1)a_{24}(i_2)
            \\
            & 
            & +a_{23}(i_1)a_{33}(i_2)
            & +a_{23}(i_1)a_{34}(i_2)
            \\
            & 
            & 
            & +a_{34}{(i_1)a_{44}(i_2)}
            \\[2ex]
            & 
            & a_{33}(i_1)a_{33}(i_2)
            & a_{33}(i_1)a_{34}(i_2)
            \\
            & 
            & 
            & +a_{34}(i_1)a_{44}(i_2)
            \\[2ex]
            & 
            &
            & a_{44}(i_1)a_{44}(i_2)
        \end{pmatrix}.
    \end{align*} 
    Hence, the formula is true for $k=2$. Suppose that it is true for $k=m$. We need to show the same holds for $k=m+1$. Note that 
    \begin{align*}
        \prod_{j=1}^{m+1} u_{i_j} & =  \left(\prod_{j=1}^m u_{i_j}\right) u_{i_{m+1}} 
        \\[1ex]
        & = \begin{pmatrix}
            \displaystyle\prod_{j=1}^{m+1} a_{11}(i_j) & \lambda'_{12} & \lambda'_{13} & \lambda'_{14} 
            \\[1ex]
            & \displaystyle\prod_{j=1}^{m+1} a_{22}(i_j) & \lambda'_{23} & \lambda'_{24} 
            \\[1ex]
            & & \displaystyle\prod_{j=1}^{m+1} a_{33}(i_j) & \lambda'_{34} 
            \\[1ex]
            & & & \displaystyle\prod_{j=1}^{m+1} a_{44}(i_j)
        \end{pmatrix}. 
    \end{align*}
    Now we will compute $\lambda'_{\alpha \beta}$. We have 
    \begin{align*}
        \lambda'_{12} & = \prod_{j=1}^{m}a_{11}(i_j) a_{12}(i_{m+1}) + \textcolor{blue}{\lambda_{12}} a_{22}(i_{m+1}) 
        \\
        & = \prod_{j=1}^{m}a_{11}(i_j) a_{12}(i_{m+1}) + \textcolor{blue}{\sum_{s=1}^m \left(\prod_{j=1}^{s-1}a_{11}(i_j)\right)a_{12}(i_s) \left(\prod_{j=s+1}^m a_{22}(i_j)\right)} a_{22}(i_{m+1}) 
        \\
        & = \sum_{s=1}^{m+1} \left(\prod_{j=1}^{s-1}a_{11}(i_j)\right)a_{12}(i_s) \left(\prod_{j=s+1}^{m+1} a_{22}(i_j)\right).
    \end{align*}
    Similarly, it can be shown for the entries $\lambda'_{23}$ and $\lambda'_{34}$. 
    \begin{align*}
        \lambda'_{13} & = \prod_{j=1}^m a_{11}(i_j)a_{13}(i_{m+1}) + \textcolor{blue}{\lambda_{12}}a_{23}(i_{m+1}) + \textcolor{newBlue}{\lambda_{13}}a_{33}(i_{m+1}) 
        \\
        & = \prod_{j=1}^m a_{11}(i_j)a_{13}(i_{m+1}) 
        \\
        & \quad + \textcolor{blue}{\sum_{s=1}^m \left(\prod_{j=1}^{s-1}a_{11}(i_j)\right)a_{12}(i_s) \left(\prod_{j=s+1}^m a_{22}(i_j)\right)}a_{23}(i_{m+1}) 
        \\
        & \quad + \textcolor{newBlue}{\sum_{s=1}^{m}\left(\prod_{j=1}^{s-1}a_{11}(i_j)\right) a_{13}(i_s) \left(\prod_{j=s+1}^{m}a_{33}(i_j)\right) }a_{33}(i_{m+1}) 
        \\
        & \quad + \textcolor{newBlue}{\sum_{1\le s<t\le m}\left(\prod_{j=1}^{s-1}a_{11}(i_j)\right) a_{12}(i_s) \left(\prod_{j=s+1}^{t-1}a_{22}(i_j)\right)a_{23}(i_t)  \left(\prod_{j=t+1}^{m}a_{33}(i_j)\right)}a_{33}(i_{m+1})
        \\
        & = \textcolor{black}{\sum_{s=1}^{m+1}\left(\prod_{j=1}^{s-1}a_{11}(i_j)\right) a_{13}(i_s) \left(\prod_{j=s+1}^{m+1}a_{33}(i_j)\right)} 
        \\
        & + \quad \textcolor{black}{\sum_{1\le s<t\le m+1}\left(\prod_{j=1}^{s-1}a_{11}(i_j)\right) a_{12}(i_s) \left(\prod_{j=s+1}^{t-1}a_{22}(i_j)\right)a_{23}(i_t)  \left(\prod_{j=t+1}^{m+1}a_{33}(i_j)\right)}.
    \end{align*}
    Similarly, it can be shown for the entry $\lambda'_{24}$. Finally,
    \begin{align*}
        \lambda'_{14} & = \prod_{j=1}^m a_{11}(i_j)a_{14}(i_{m+1}) + \textcolor{blue}{\lambda_{12}} a_{24}(i_{m+1}) + \textcolor{newBlue}{\lambda_{13}} a_{34}(i_{m+1}) + \textcolor{teal}{\lambda_{14}} a_{44}(i_{m+1})
        \\
        & = \prod_{j=1}^m a_{11}(i_j)a_{14}(i_{m+1}) 
        \\
        & \quad + \textcolor{blue}{\sum_{s=1}^m \left(\prod_{j=1}^{s-1}a_{11}(i_j)\right)a_{12}(i_s) \left(\prod_{j=s+1}^m a_{22}(i_j)\right)}a_{24}(i_{m+1})
        \\
        & \quad + \textcolor{newBlue}{\sum_{s=1}^{m}\left(\prod_{j=1}^{s-1}a_{11}(i_j)\right) a_{13}(i_s) \left(\prod_{j=s+1}^{m}a_{33}(i_j)\right) }a_{34}(i_{m+1}) 
        \\
        & \quad + \textcolor{newBlue}{\sum_{1\le s<t\le m}\left(\prod_{j=1}^{s-1}a_{11}(i_j)\right) a_{12}(i_s) \left(\prod_{j=s+1}^{t-1}a_{22}(i_j)\right)a_{23}(i_t)  \left(\prod_{j=t+1}^{m}a_{33}(i_j)\right)}a_{34}(i_{m+1})
        \\
        & \quad + \textcolor{teal}{\sum_{s=1}^{k}\left(\prod_{j=1}^{s-1}a_{11}(i_j)\right) a_{14}(i_s) \left(\prod_{j=s+1}^{k}a_{44}(i_j)\right)} a_{44}(i_{m+1})
        \\[1ex]
        & \quad +  \textcolor{teal}{\sum_{1\le s<t\le k}\left(\prod_{j=1}^{s-1}a_{11}(i_j)\right) a_{12}(i_s) \left(\prod_{j=s+1}^{t-1}a_{22}(i_j)\right)a_{24}(i_t)  \left(\prod_{j=t+1}^{k}a_{44}(i_j)\right)}a_{44}(i_{m+1})
        \\[1ex]
        & \quad +  \textcolor{teal}{\sum_{1\le s<t\le k}\left(\prod_{j=1}^{s-1}a_{11}(i_j)\right) a_{13}(i_s) \left(\prod_{j=s+1}^{t-1}a_{33}(i_j)\right)a_{34}(i_t)  \left(\prod_{j=t+1}^{k}a_{44}(i_j)\right)}a_{44}(i_{m+1})
        \\[1ex]
        & \quad +  \textcolor{teal}{\sum_{1\le s<t<u\le k}\left[\left(\prod_{j=1}^{s-1}a_{11}(i_j)\right) a_{12}(i_s) \left(\prod_{j=s+1}^{t-1}a_{22}(i_j)\right)a_{23}(i_t)  \left(\prod_{j=t+1}^{u-1}a_{33}(i_j)\right)\right.}
        \\[1ex]
        & \kern 3cm \textcolor{teal}{\left.a_{34}(i_u)  \left(\prod_{j=u+1}^{k}a_{44}(i_j)\right)\right]}a_{44}(i_{m+1})
        \\
        & = \sum_{s=1}^{m+1}\left(\prod_{j=1}^{s-1}a_{11}(i_j)\right) a_{14}(i_s) \left(\prod_{j=s+1}^{m+1}a_{44}(i_j)\right) 
        \\[1ex]
        & +  \sum_{1\le s<t\le m+1}\left(\prod_{j=1}^{s-1}a_{11}(i_j)\right) a_{12}(i_s) \left(\prod_{j=s+1}^{t-1}a_{22}(i_j)\right)a_{24}(i_t)  \left(\prod_{j=t+1}^{m+1}a_{44}(i_j)\right)
        \\[1ex]
        & +  \sum_{1\le s<t\le m+1}\left(\prod_{j=1}^{s-1}a_{11}(i_j)\right) a_{13}(i_s) \left(\prod_{j=s+1}^{t-1}a_{33}(i_j)\right)a_{34}(i_t)  \left(\prod_{j=t+1}^{m+1}a_{44}(i_j)\right)
        \\[1ex]
        & +  \sum_{1\le s<t<u\le m+1}\left[\left(\prod_{j=1}^{s-1}a_{11}(i_j)\right) a_{12}(i_s) \left(\prod_{j=s+1}^{t-1}a_{22}(i_j)\right)a_{23}(i_t)  \left(\prod_{j=t+1}^{u-1}a_{33}(i_j)\right)\right.
        \\[1ex]
        & \kern 2cm \left.a_{34}(i_u)  \left(\prod_{j=u+1}^{m+1}a_{44}(i_j)\right)\right].
    \end{align*}
    This proves the lemma.
\end{proof}

Using the previous lemma, it is easy to describe the entries of polynomial images, as it is evident from the following result. The next lemma will prove how does $p \left(\vec{\mathbf{u}}\right)$ look like where $\vec{\mathbf{u}}$ is $\left(u_1,\ldots,u_n\right)$.

\begin{lemma}\label{lem:polynomialMap}
    Let
    \begin{displaymath}
        u_i = 
              \begin{pmatrix}
                  a_{11}(i) & a_{12}(i) & a_{13}(i) & a_{14}(i) \\
                            & a_{22}(i) & a_{23}(i) & a_{24}(i) \\
                            &           & a_{33}(i) & a_{34}(i) \\
                            &           &           & a_{44}(i)
              \end{pmatrix} \in \tk,
   \end{displaymath}
   for $i=1,\ldots,n$. Then, $p \left(\vec{\mathbf{u}}\right)$ is given by 
   \begin{align*}
        \begin{pmatrix}
            p\left(\avect{11}\right) 
            & \displaystyle\sum_i p_i \left(\avec{1,2}\right)(i) 
            & \displaystyle\sum_i p_i \left(\avec{1,3}\right)(i)
            & \displaystyle\sum_i p_i \left(\avec{1,4}\right)(i)
            \\
            & 
            & + \displaystyle\sum_{i,j}p_{ij}\left(\avec{1,2,3}\right)(i,j)
            & + \displaystyle\sum_{i,j}p_{ij}\left(\avec{1,2,4}\right)(i,j)
            \\ 
            & 
            & 
            & + \displaystyle\sum_{i,j}p_{ij}\left(\avec{1,3,4}\right)(i,j)
            \\ 
            &
            & 
            & + \displaystyle\sum_{i,j,k}p_{ijk}\left(\avec{1,2,3,4}\right)(i,j,k)
            \\[5ex]
            &  p\left(\avect{22}\right) 
            & \displaystyle\sum_i p_i \left(\avec{2,3}\right)(i) 
            & \displaystyle\sum_i p_i \left(\avec{2,4}\right)(i) 
            \\
            &
            & 
            & +\displaystyle\sum_{i,j}p_{ij}\left(\avec{2,3,4}\right)(i,j)
            \\[5ex]
            &   
            & p\left(\avect{33}\right) 
            & \displaystyle\sum_i p_i \left(\avec{3,4}\right)(i) 
            \\[5ex]
            & 
            & 
            &  p\left(\avect{44}\right)
        \end{pmatrix},
   \end{align*}
   where $\vec{\mathbf{u}} = \left(u_1,\ldots,u_n\right) $  $\avect{ii} = \left(a_{ii}(1),\ldots, a_{ii}(n)\right)$ and
   \begin{displaymath}
    p_{ijk}\left(\avec{\alpha,\beta,\gamma, \delta}\right)(i,j,k) = p_{ijk}\left(\avect{\alpha \alpha},\avect{\beta \beta}, \avect{\gamma \gamma}, \avect{\delta \delta}\right)a_{\alpha \beta}(i)a_{\beta \gamma}(j)a_{\gamma \delta}(k).
   \end{displaymath} 
   Here $p_{ijk}$ is a polynomial over $K$ for $i,j=1,2,\ldots, n$. 
\end{lemma}

\begin{proof}
    Recall that the form of $p$ is given by \autoref{eq:generalPolynomial}.  Using \autoref{lem:product}, we have
    \begin{align*}
        p \left(\vec{\mathbf{u}}\right) & = \sum_{k=1}^{d} \left(\sum_{1\leq i_1,\ldots,i_k \leq n} \lambdak u_{i_1} \cdots u_{i_k}\right)
        \\[1ex]
        & = \sum_{k=1}^d \left(\sum_{1\leq i_1,\ldots,i_k \leq n} \lambdak\Lambda_k\right) 
        \\[1ex]
        & = 
        \begin{pmatrix}
            p(\avect{11}) & a'_{12} & a'_{13} & a'_{14} 
            \\
            0             & p(\avect{22}) & a'_{23} & a'_{24} 
            \\
            0             & 0       & p(\avect{33}) & a'_{34} 
            \\
            0             & 0       & 0       & p(\avect{44}) 
            \\
        \end{pmatrix},
    \end{align*}
    where
    \begin{align*}
        a'_{l l+1} & = \sum_{k=1}^d \left(\sum_{1\le i_1,\ldots,i_k\leq n}\lambdak \lambda'_{l l+1} \right)
        \\
        & = \sum_{k=1}^d \left(\sum_{1\le i_1,\ldots,i_k\leq n}\lambdak \sum_{s=1}^{m+1} \left(\prod_{j=1}^{s-1}a_{l l}(i_j)\right)a_{l l+1}(i_s) \left(\prod_{j=s+1}^{m+1} a_{l+1 l+1}(i_j)\right)\right)
        \\
        & = \sum_{i=1}^n p_i \left(\avect{ll},\avect{l+1 l+1}\right)a_{l l+1}(i),
    \end{align*}
    where $p_i \in K \left[\mathbf{x}_1,\mathbf{x}_2\right]$ is a polynomial over $K$ for $i=1,\ldots,n$.
    \begin{align*}
        a'_{ll+2} & =  \sum_{k=1}^d \left(\sum_{1\le i_1,\ldots,i_k\leq n}\lambdak \lambda'_{ll+2} \right)
        \\
        & = \sum_{k=1}^d \left(\sum_{1\le i_1,\ldots,i_k\leq n}\lambdak \left[\sum_{s=1}^{m+1}\left(\prod_{j=1}^{s-1}a_{l l}(i_j)\right) a_{l l+2}(i_s) \left(\prod_{j=s+1}^{m+1}a_{l+2 l+2}(i_j)\right)\right.\right.
        \\
        & + \sum_{1\le s<t\le m+1}\left(\prod_{j=1}^{s-1}a_{11}(i_j)\right) a_{l l+1}(i_s) \left(\prod_{j=s+1}^{t-1}a_{l+1 l+1 }(i_j)\right)a_{l+1 l+2}(i_t) 
        \\
        &  \kern 5cm \left.\left. \left(\prod_{j=t+1}^{m+1}a_{l+2 l+2}(i_j)\right)\right]\right)
        \\
        & = \sum_{i=1}^n p_i \left(\avect{l l},\avect{l+2 l+2}\right)a_{l l+1}(i)  + \sum_{i,j=1}^n p_{ij}\left(\avect{l l},\avect{l+1 l+1},\avect{l+2 l+2}\right)a_{l l+1}(i) a_{l+1 l+2}(j),
    \end{align*}
    where $p_{ij}$ is a polynomial in $K \left[\mathbf{x}_1,\mathbf{x}_2,\mathbf{x}_3\right]$ for $i,j=1,\ldots,n$. Finally, 
    \begin{align*}
        a_{14}' & =  \sum_{k=1}^d \left(\sum_{1\le i_1,\ldots,i_k\leq n}\lambdak \lambda'_{14} \right)
        \\
        & = \sum_{k=1}^d \left(\sum_{1\le i_1,\ldots,i_k\leq n}\lambdak \left[\sum_{s=1}^{m+1}\left(\prod_{j=1}^{s-1}a_{11}(i_j)\right) a_{14}(i_s) \left(\prod_{j=s+1}^{m+1}a_{44}(i_j)\right)\right.\right.
        \\
        & + \sum_{1\le s<t\le m+1}\left(\prod_{j=1}^{s-1}a_{11}(i_j)\right) a_{12}(i_s) \left(\prod_{j=s+1}^{t-1}a_{22}(i_j)\right)a_{24}(i_t)  \left(\prod_{j=t+1}^{m+1}a_{44}(i_j)\right)
        \\
        & + \sum_{1\le s<t\le m+1}\left(\prod_{j=1}^{s-1}a_{11}(i_j)\right) a_{13}(i_s) \left(\prod_{j=s+1}^{t-1}a_{33}(i_j)\right)a_{34}(i_t)  \left(\prod_{j=t+1}^{m+1}a_{44}(i_j)\right) 
        \\
        & + \sum_{1\le s<t<u\le m+1}\left[\left(\prod_{j=1}^{s-1}a_{11}(i_j)\right) a_{12}(i_s) \left(\prod_{j=s+1}^{t-1}a_{22}(i_j)\right)a_{23}(i_t)  \left(\prod_{j=t+1}^{u-1}a_{33}(i_j)\right)\right.
        \\
        & \left.\left.\kern 2cm \left.a_{34}(i_u)  \left(\prod_{j=u+1}^{m+1}a_{44}(i_j)\right)\right]\right]\right)
        \\
        & = \sum_{i=1}^n p_i \left(\avect{11},\avect{22}\right)a_{12}(i)  + \sum_{i,j=1}^n p_{ij}\left(\avect{11},\avect{22},\avect{33}\right)a_{12}(i) a_{23}(j)
        \\
        & + \sum_{i,j=1}^n p_{ij}\left(\avect{11},\avect{22},\avect{33}\right)\avect{44}a_{12}(i) a_{23}(j)a_{34},
    \end{align*}
    where $p_{ijk} \in K \left[\mathbf{x}_1,\mathbf{x}_2,\mathbf{x}_3\right]$ is a polynomial for $i,j,k=1,\ldots,n$.
\end{proof}

Now suppose that we are given with a matrix which has a nonzero entry at a particular place. Then depending upon the order of $p$ we will find either of $p_i,~p_{ij},$ or $p_{ijk}$ which is nonzero. To get an arbitrary image with certain conditions we need to manipulate the images of these polynomials, to show the candidacy of the matrix inside the desired image set. The following technical lemma will be helpful in this regard.

\begin{lemma}\label{lem:polyNonZero}
    Let $p\in K[F_{3n}]$ be a polynomial over a field $K$. Suppose there exists $\vec{a},\vec{b},\vec{d}\in K^n$ such that $p \left(\vec{a},\vec{b},\vec{d}\right)\neq 0$, then there exist $\vec{a}',\vec{b}',\vec{c}',\vec{d}'\in K^n$ such that 
    \begin{gather}\label{eq:polyNonZero}
        \begin{aligned}
            & p \left(\vec{a}',\vec{b}',\vec{c}'\right) \neq 0 \\
            & p \left(\vec{a}',\vec{b}',\vec{d}'\right) \neq 0 \\
            & p \left(\vec{a}',\vec{c}',\vec{d}'\right) \neq 0, \text{ and }\\
            & p \left(\vec{b}',\vec{c}',\vec{d}'\right) \neq 0.
        \end{aligned}
    \end{gather}
\end{lemma}
\begin{proof}
    Since $p \left(\vec{a},\vec{b},\vec{d}\right)\neq 0$ and hence $p(x,y,z)$ is a non-zero polynomial. Define 
    \begin{displaymath}
        q(x,y,z,w) = p(x,y,z) \cdot p(x,y,w) \cdot p(x,z,w) \cdot p(y,z,w).
    \end{displaymath}  
    Since $p(x,y,z), p(x,y,w), p(x,z,w)$ and $p(y,z,w)$ are non-zero polynomials and $K$ is a field, $q(x,y,z,w)$ will be a nonzero polynomial. Hence, there exists $\vec{a}',\vec{b}',\vec{c}',\vec{d}'\in K^n$ such that \autoref{eq:polyNonZero} holds.
\end{proof}

We are now ready to prove our main theorem of this section.

\begin{thm}\label{thm:mainThmFor_n=4}
    Suppose that $p(x_1,\ldots,x_n)$ is a polynomial with zero constant term in non commutative variables over an algebraically closed field $K$. Then one of the following statements must holds:
    \begin{enumerate}[(i)]
        \item If $\ord \ge 4$, then $p \left(\tk\right) = \{0\}$;
        \item If $\ord  = 3$, then $p \left(\tk\right)  = \tk^{(2)}$;
        \item If $\ord  = 2$, then $p \left(\tk\right)  = \tk^{(1)}$;    
        \item If $\ord  = 1$, then $p \left(\tk\right)  = \tk^{(0)}$.
        \item If $\ord  = 0$, then $p \left(\tk\right)$ is Zariski dense in $T_4(K)$.
    \end{enumerate}
\end{thm}

\begin{proof}
    \begin{enumerate}[(i)]
        \item Suppose that $\ord \ge 4$. This implies, from the definition, that $p \left(\tk\right)=\{0\}$ and hence the statement (i) is proved.
       
        \item Suppose that $\ord=3$. This implies that $p\left(\tk[3]\right)=0$ but $p \left(\tk\right)\neq 0$. Let 
        \begin{displaymath}
             u_i = 
                   \begin{pmatrix}
                       a_{11}(i) & a_{12}(i) & a_{13}(i) & a_{14}(i) \\
                                 & a_{22}(i) & a_{23}(i) & a_{24}(i) \\
                                 &           & a_{33}(i) & a_{34}(i) \\
                                 &           &           & a_{44}(i)
                   \end{pmatrix} \in \tk,
        \end{displaymath}
        for $i=1,\ldots,n$. Then using \autoref{lem:polynomialMap}, \cite{WaZhLu21} ,and \cite[Lemma 3.1]{WaZhLu22}, we have
        \begin{displaymath}
            p \left(\vec{\mathbf{u}}\right) = 
            \begin{pmatrix}
                0 & 0 & 0 & \displaystyle\sum_{i,j,k}p_{ijk}\left(\avec{1,2,3,4}\right)(i,j,k)
                \\[2ex]
                & 0 & 0 & 0 \\
                &   & 0 & 0 \\
                &   &   & 0
            \end{pmatrix}.
        \end{displaymath}
        This implies that $p \left(\tk\right)\subseteq \tk^{(2)}$. Now we need to show that $\tk^{(2)}\subseteq p \left(\tk\right)$. Let 
        \begin{displaymath}
            X = 
             \begin{pmatrix}
                0 & 0 & 0 & x_{14} \\
                  & 0 & 0 & 0 \\
                  &   & 0 & 0 \\
                  &   &   & 0    
             \end{pmatrix} \in \tk^{(2)}. 
        \end{displaymath}  
        Since $p \left(\tk\right)\neq 0$, we can find $i_0,j_0,k_0\in \{1,\ldots,n\}$ such that 
        \begin{displaymath}
             p_{i_0j_0k_0}\left(\avec{1,2,3,4}\right)\neq 0,
        \end{displaymath} 
        for some $\avect{11},\avect{22},\avect{33},\avect{44}\in K^n$. We set
        \begin{displaymath}
             \lambda_0 \defeq p_{i_0j_0k_0}\left(\avec{1,2,3,4}\right).
        \end{displaymath}
        Suppose first that $i_0=j_0=k_0$. Take 
        \begin{displaymath}
             u_i =  \begin{cases}
                        \begin{pmatrix}
                            a_{11}(i) & 0         & 0         & 0 \\
                                        & a_{22}(i) & 0         & 0 \\
                                        &           & a_{33}(i) & 0 \\
                                        &           &           & a_{44}(i)
                        \end{pmatrix}, & \text{ if } i \neq i_0 
                        \\[1.5cm]
                        \begin{pmatrix}
                            a_{11}(i_0) & \lambda_0^{-1}x_{14}   & 0         & 0 \\
                                        & a_{22}(i_0)            & 1         & 0 \\
                                        &               & a_{33}(i_0) & 1 \\
                                        &               &           & a_{44}(i_0)
                        \end{pmatrix}, & \text{ if } i=i_0. 
                    \end{cases}
        \end{displaymath}
        Note that 
        \begin{displaymath}
            p \left(\vec{\mathbf{u}} \right) = 
            \begin{pmatrix}
                0 & 0 & 0 & p_{i_0j_0k_0}\left(\avec{1,2,3,4}\right)\lambda_0^{-1}x_{14} \cdot 1\cdot 1
                \\[2ex]
                & 0 & 0 & 0 \\
                &   & 0 & 0 \\
                &   &   & 0
            \end{pmatrix} = X.
        \end{displaymath}
        Thus, $p \left(\tk\right)\subseteq  \tk^{(2)}$. Hence, we proved that $p \left(\tk\right) = \tk^{(2)}$. Now we assume that $i_0\neq j_0=k_0$. Take
        \begin{displaymath}
             u_i = 
             \begin{cases}
                \begin{pmatrix}
                    a_{11}(i) & 0         & 0         & 0 \\
                                & a_{22}(i) & 0         & 0 \\
                                &           & a_{33}(i) & 0 \\
                                &           &           & a_{44}(i)
                \end{pmatrix}, & \text{ if } i \neq i_0,j_0 
                \\[1.5cm]
                \begin{pmatrix}
                    a_{11}(i_0) & \lambda_0^{-1}x_{14}   & 0         & 0 \\
                                & a_{22}(i_0)            & 0         & 0 \\
                                &               & a_{33}(i_0) &  \\
                                &               &           & a_{44}(i_0)
                \end{pmatrix}, & \text{ if } i=i_0
                \\[1.5cm]
                \begin{pmatrix}
                    a_{11}(j_0) & 0   & 0         & 0 \\
                                & a_{22}(j_0)            & 1         & 0 \\
                                &               & a_{33}(j_0) & 1 \\
                                &               &           & a_{44}(j_0)
                \end{pmatrix}, & \text{ if } i=j_0.
             \end{cases}
        \end{displaymath}
        So we have
        \begin{displaymath}
            p \left(\vec{\mathbf{u}} \right) = 
            \begin{pmatrix}
                0 & 0 & 0 & p_{i_0j_0k_0}\left(\avec{1,2,3,4}\right)\lambda_0^{-1}x_{14} \cdot 1\cdot 1
                \\[2ex]
                & 0 & 0 & 0 \\
                &   & 0 & 0 \\
                &   &   & 0
            \end{pmatrix} = X.
        \end{displaymath}
        Hence, $p \left(\tk\right) = \tk^{(2)}$. If $i_0\neq j_0\neq k_0$, then take 
        \begin{displaymath}
            u_{j_0} = 
            \begin{pmatrix}
                a_{11}(j_0) & 0   & 0         & 0 \\
                            & a_{22}(j_0)            & 1         & 0 \\
                            &               & a_{33}(j_0) & 0 \\
                            &               &           & a_{44}(j_0)
            \end{pmatrix}, \text{ and }
            u_{k_0} = 
            \begin{pmatrix}
                a_{11}(k_0) & 0   & 0         & 0 \\
                            & a_{22}(k_0)            &          & 0 \\
                            &               & a_{33}(k_0) & 1 \\
                            &               &           & a_{44}(k_0)
            \end{pmatrix}
        \end{displaymath}
        and remaining $u_i's$ are same. Hence, $p \left(\tk\right) = \tk^{(2)}$.

        \item Suppose that $\ord=2$. This implies that $p\left(\tk[2]\right)=0$ but $p \left(\tk[3]\right)\neq 0$. Let 
        \begin{displaymath}
             u_i = 
                   \begin{pmatrix}
                       a_{11}(i) & a_{12}(i) & a_{13}(i) & a_{14}(i) \\
                                 & a_{22}(i) & a_{23}(i) & a_{24}(i) \\
                                 &           & a_{33}(i) & a_{34}(i) \\
                                 &           &           & a_{44}(i)
                   \end{pmatrix} \in \tk,
        \end{displaymath}
        for $i=1,\ldots,n$. Then using \autoref{lem:polynomialMap} and \cite{WaZhLu21}, we have
        \begin{displaymath}
            p \left(\vec{\mathbf{u}}\right) = 
            \begin{pmatrix}
                0 
                & 0 
                & \displaystyle \sum_{i,j} p_{ij}\left(\avec{1,2,3}\right)(i,j)
                & \displaystyle\sum_{i,j}p_{ij}\left(\avec{1,2,4}\right)(i,j)
                \\
                & 
                & 
                & + \displaystyle\sum_{i,j}p_{ij}\left(\avec{1,3,4}\right)(i,j)
                \\\
                & 
                & 
                & + \displaystyle\sum_{i,j,k}p_{ijk}\left(\avec{1,2,3,4}\right)(i,j,k)
                \\[3ex]
                & 0 
                & 0 
                & \displaystyle\sum_{i,j}p_{ij}\left(\avec{2,3,4}\right)(i,j)
                \\
                &   
                & 0 
                & 0 
                \\
                &   
                &   
                & 0
            \end{pmatrix}.
        \end{displaymath}
    This implies that $p \left(\tk\right)\subseteq \tk^{(1)}$. For showing the other side inclusion, take
    \begin{displaymath}
        X = 
        \begin{pmatrix}
            0 & 0 & x_{13} & x_{14} \\
              & 0 &   0    & x_{24} \\
              &   &   0    &    0   \\
              &   &        &    0
        \end{pmatrix}\in T_4(K)^{(1)}.
    \end{displaymath} 
    Since $p \left(T_3(K)\right)\neq 0$, there exist $i_0,j_0\in \left\{1,\ldots,n\right\}$ such that 
    \begin{displaymath}
        p_{i_0j_0} \left(\avect{1,2,4}\right)\neq 0
    \end{displaymath} 
    for some $\avect{11},\avect{22},\avect{44}\in K^n$. By using \autoref{lem:polyNonZero} we can find $\avect{11}',\avect{22}',\avect{33}',\avect{44}'\in K^n$ such that 
    \begin{align*}
        p_{i_0j_0}\left(\avec{1,2,3}'\right) \neq 0 \\
        p_{i_0j_0}\left(\avec{1,2,4}'\right) \neq 0 \\ 
        p_{i_0j_0}\left(\avec{1,3,4}'\right) \neq 0 \\
        p_{i_0j_0}\left(\avec{2,3,4}'\right) \neq 0.
    \end{align*}
    We set 
    \begin{align*}
        \lambda_1 = p_{i_0j_0}\left(\avec{1,2,4}'\right);\\
        \lambda_2 = p_{i_0j_0}\left(\avec{1,2,3}'\right);\\
        \lambda_3 = p_{i_0j_0}\left(\avec{2,3,4}'\right);\\
        \lambda_4 = p_{i_0j_0}\left(\avec{1,3,4}'\right).
    \end{align*}
    We first assume that $i_0=j_0$. Furthermore, we divide our proof into a couple of cases. 
    \bigskip
    
    \noindent \textbf{Case I: }$x_{13} \neq 0$. We take 
    \begin{displaymath}
        u_i = 
        \begin{cases}
            \diag \left(a_{11}(i),a_{22}(i),a_{33}(i),a_{44}(i)\right), & i\neq i_0
            \\[2ex]
            \begin{pmatrix}
                a_{11}\left(i\right) & \lambda_2^{-1}x_{13} & 0 & 0 
                \\
                & a_{22} \left(i\right) & 1 & \lambda_1^{-1}\lambda_2 x_{13}^{-1}\left(x_{14}-\alpha \lambda_2^{-1}\lambda_3^{-1} x_{13}\textcolor{black}{x_{24}} \right) 
                \\
                &   & a_{33}\left(i\right) & \lambda_3^{-1}\textcolor{black}{x_{24}}
                \\
                &   &   & a_{44}\left(i\right)
            \end{pmatrix}, & i=i_0
        \end{cases},
    \end{displaymath} 
    where $\alpha = p_{i_0i_0i_0} \left(\avec{1,2,3,4}'\right)$. Then we have
    \begin{align*}
        p \left({\vec{\mathbf{u}}}\right)_{13} & = p_{i_0i_0}(\avec{1,2,3})a_{12}(i_0)a_{23}(i_0)
        \\
        & = \lambda_2 \cdot \lambda_2^{-1}x_{13},
        \\
        & = x_{13}
        \\
        p \left({\vec{\mathbf{u}}}\right)_{24} & = p_{i_0i_0}(\avec{2,3,4})a_{23}(i_0)a_{34}(i_0)
        \\
        & = \lambda_3 \cdot 1 \cdot \lambda_3^{-1}x_{24} 
        \\
        & = x_{24}
        \\
        p \left({\vec{\mathbf{u}}}\right)_{14} & = p_{i_0i_0}(\avec{1,2,4})a_{12}(i_0)a_{24}(i_0) + p_{i_0i_0}(\avec{1,3,4})a_{13}(i_0)a_{34}(i_0)
        \\ 
        & + p_{i_0i_0i_0}(\avec{1,2,3,4})a_{12}(i_0)a_{23}(i_0)a_{34}(i_0)
        \\
        & = \lambda_1 \lambda_2^{-1}x_{13} \lambda_1^{-1}\lambda_2 x_{13}^{-1}\left(x_{14}-\alpha \lambda_2^{-1}\lambda_3^{-1} x_{13}\textcolor{black}{x_{24}} \right) + 0 
        \\
        & + \alpha \lambda_2^{-1}x_{13}\cdot 1 \cdot \lambda_3^{-1}x_{24}
        \\
        & = x_{14}
    \end{align*}

    \noindent \textbf{Case II: }$x_{13} = 0$. Take 
    \begin{displaymath}
        u_i = 
        \begin{cases}
            \diag \left(a_{11}(i),a_{22}(i),a_{33}(i),a_{44}(i)\right), & i\neq i_0
            \\[2ex]
            \begin{pmatrix}
                a_{11}\left(i\right) & 0 & \lambda_4^{-1}x_{14} & 0 
                \\
                & a_{22} \left(i\right) & \lambda_3^{-1}\textcolor{black}{x_{24}} & 0 
                \\
                &   & a_{33}\left(i\right) & 1 \\
                &   &   & a_{44}\left(i\right)
            \end{pmatrix}, & i=i_0
        \end{cases}.
    \end{displaymath} 
    In this case also note that $p \left(\vec{\mathbf{u}}\right)=X$. This proves that $T_4(K)^{(1)}\subseteq p \left(T_4(k)\right)$.
    \bigskip

    \noindent Now we assume that $i_0\neq j_0$, and again we divide our proof into a couple of cases.

    \noindent \textbf{Case I: }$x_{13} \neq 0$. We can assume that $p_{i i}$ are identically zero polynomials for all $i$. Otherwise, it will reduce to the $i=j$ case. Now we take 
    \begin{displaymath}
        u_i = 
        \begin{cases}
            \diag \left(a_{11}(i),a_{22}(i),a_{33}(i),a_{44}(i)\right), & i\neq i_0, j_0
            \\[2ex]
            \begin{pmatrix}
                a_{11}\left(i\right) & \lambda_2^{-1}x_{13} & 0 & 0 
                \\
                & a_{22} \left(i\right) & 1 & 0 
                \\
                &   & a_{33}\left(i\right) & 0 
                \\
                &   &   & a_{44}\left(i\right)
            \end{pmatrix}, & i=i_0
            \\[3ex]
            \begin{pmatrix}
                a_{11}\left(i\right) & 0 & 0 & 0 
                \\
                & a_{22} \left(i\right) & 1 & \gamma 
                \\
                &   & a_{33}\left(i\right) & \lambda_3^{-1}x_{24} 
                \\
                &   &   & a_{44}\left(i\right)
            \end{pmatrix}, & i=j_0
        \end{cases},
    \end{displaymath} 
    where $\gamma = \lambda_1^{-1} \lambda_2 x_{13}^{-1}(x_{14}-\beta _{14})$, for a suitable choice of $\beta _{14}$. Then we have with the given choice that $p \left(\mathbf{\vec{u}}\right)=X$. 

    \noindent \textbf{Case II: }$x_{13} = 0$. Take 
    \begin{displaymath}
        u_i = 
        \begin{cases}
            \diag \left(a_{11}(i),a_{22}(i),a_{33}(i),a_{44}(i)\right), & i\neq i_0
            \\[2ex]
            \begin{pmatrix}
                a_{11}\left(i\right) & 0 & \lambda_4^{-1}x_{14} & 0 \\
                    & a_{22} \left(i\right) & \lambda_3^{-1}x_{24} & 0 \\
                    &   & a_{33}\left(i\right) & 0 \\
                    &   &   & a_{44}\left(i\right)
            \end{pmatrix}, & i=i_0
            \\[2ex]
            \begin{pmatrix}
                a_{11}\left(i\right) & 0 & 0 & 0 \\
                    & a_{22} \left(i\right) & 0 & 0 \\
                    &   & a_{33}\left(i\right) & 1 \\
                    &   &   & a_{44}\left(i\right)
            \end{pmatrix}, & i=j_0
        \end{cases}.
    \end{displaymath} 
    Thus, we have $p \left(\vec{\mathbf{u}}\right)=X$ and hence $T_4(K)^{(1)}\subseteq p \left(T_4(k)\right)$.

        \item Suppose that $\ord=1$. This implies that $p\left(\tk[1]\right)=0$ but $p \left(\tk[2]\right)\neq 0$. Let 
        \begin{displaymath}
             u_i = 
                   \begin{pmatrix}
                       a_{11}(i) & a_{12}(i) & a_{13}(i) & a_{14}(i) \\
                                 & a_{22}(i) & a_{23}(i) & a_{24}(i) \\
                                 &           & a_{33}(i) & a_{34}(i) \\
                                 &           &           & a_{44}(i)
                   \end{pmatrix} \in \tk,
        \end{displaymath}
        for $i=1,\ldots,n$. We have
        \begin{displaymath}
            p \left(\vec{\mathbf{u}}\right) = 
            \begin{pmatrix}
                0
                & \displaystyle\sum_i p_i \left(\avec{1,2}\right)(i) 
                & \displaystyle\sum_i p_i \left(\avec{1,3}\right)(i)
                & \displaystyle\sum_i p_i \left(\avec{1,4}\right)(i)
                \\
                & 
                & + \displaystyle\sum_{i,j}p_{ij}\left(\avec{1,2,3}\right)(i,j)
                & + \displaystyle\sum_{i,j}p_{ij}\left(\avec{1,2,4}\right)(i,j)
                \\ 
                & 
                & 
                & + \displaystyle\sum_{i,j}p_{ij}\left(\avec{1,3,4}\right)(i,j)
                \\ 
                &
                & 
                & + \displaystyle\sum_{i,j,k}p_{ijk}\left(\avec{1,2,3,4}\right)(i,j,k)
                \\[5ex]
                & 0 
                & \displaystyle\sum_i p_i \left(\avec{2,3}\right)(i) 
                & \displaystyle\sum_i p_i \left(\avec{2,4}\right)(i) 
                \\
                &
                & 
                & +\displaystyle\sum_{i,j}p_{ij}\left(\avec{2,3,4}\right)(i,j)
                \\[5ex]
                &   
                & 0
                & \displaystyle\sum_i p_i \left(\avec{3,4}\right)(i) 
                \\[5ex]
                & 
                & 
                &  0
            \end{pmatrix},
        \end{displaymath}
    This implies that $p \left(\tk\right)\subseteq \tk^{(0)}$. For showing the other side inclusion, take
    \begin{displaymath}
        X = 
        \begin{pmatrix}
            0 & x_{12} & x_{13} & x_{14} \\
              & 0      &   x_{23}    & x_{24} \\
              &        &   0         & x_{34}   \\
              &        &             &    0
        \end{pmatrix}\in \tk^{(0)}.
    \end{displaymath} 
    Since $p \left(T_3(K)\right)\neq 0$, there exist $i_0\in \left\{1,\ldots,n\right\}$ such that 
    \begin{displaymath}
        p_{i_0} \left(\avect{1,4}\right)\neq 0
    \end{displaymath} 
    for some $\avect{11},\avect{44}\in K^n$. By using a similar argument given in the \autoref{lem:polyNonZero} we can find $\avect{11}',\avect{22}',\avect{33}',\avect{44}'\in K^n$ such that 
    \begin{align*}
        & p_{i_0}\left(\avec{1,2}'\right) \neq 0,~ p_{i_0}\left(\avec{1,3}'\right) \neq 0,~p_{i_0}\left(\avec{2,3}'\right) \neq 0, 
        \\
        & p_{i_0}\left(\avec{1,4}'\right) \neq 0,~p_{i_0}\left(\avec{2,4}'\right) \neq 0 \text{ and } p_{i_0}\left(\avec{3,4}'\right) \neq 0.
    \end{align*}
    We set 
    \begin{align*}
        \lambda_{\alpha \beta} = p_{i_0} \left(\avec{\alpha,\beta}'\right).
    \end{align*}
    Take 
    \begin{displaymath}
        u_{i_0} = 
        \begin{pmatrix}
            0 
            & \lambda_{12}^{-1}x_{12} 
            & \lambda_{13}^{-1} \left(x_{13}-\alpha_{123}\lambda_{12}^{-1}\lambda_{23}^{-1}x_{12}x_{23}\right)
            & \lambda_{14}^{-1} \left(x_{14}-\alpha_{124}\lambda_{12}^{-1}x_{12} \right.
            \\
            
            &
            &
            & \cdot \lambda_{24}^{-1} \left(x_{24}-\alpha_{234}\lambda_{23}^{-1}\lambda_{34}^{-1}x_{23}x_{34}\right)
            \\
            
            &
            &
            & -\alpha_{134}\left(\lambda_{13}^{-1} \left(x_{13}-\alpha_{123}\lambda_{12}^{-1}\lambda_{23}^{-1}x_{12}x_{23}\right)\right.
            \\
            
            &
            &
            & \cdot \left.\lambda_{34}^{-1}x_{34} \right)-\alpha_{1234}\left(\lambda_{12}^{-1}x_{12}\lambda_{23}^{-1}x_{23}\right.
            \\
            
            &
            &
            & \left.\left.\lambda_{34}^{-1}x_{34}\right)\right)
            \\[2ex]
            
            & 0
            & \lambda_{23}^{-1}x_{23} 
            & \lambda_{24}^{-1} \left(x_{24}-\alpha_{234}\lambda_{23}^{-1}\lambda_{34}^{-1}x_{23}x_{34}\right)
            \\[2ex]
            
            & 
            & 0
            & \lambda_{34}^{-1}x_{34} 
            \\[2ex]
            
            & 
            & 
            & 0
            \\
        \end{pmatrix}
    \end{displaymath} 
    and
    \begin{displaymath}
        u_i = \diag \left(a_{11}(i),a_{22}(i),a_{33}(i),a_{44}(i)\right).
    \end{displaymath} 
    Then $p(\vec{\mathbf{u}})=X$ and hence the conclusion is proved.
    \item This result is proved for general $m$ in \autoref{lem:ZariskiDense}. We will not mention it here for refraining us from repeating arguments.
    \end{enumerate}
\end{proof}
\section{Multi index \texorpdfstring{$p$}{p}-inductive polynomials}\label{sec:multiIndex}

To generalize the \autoref{lem:product} in case of upper triangular matrix of order $n$, we give the following algorithm. Let $u_i\in T_n(R)$ where $R$ is a commutative polynomial algebra with unity and $\lambda_{\alpha \beta}$ denotes the $(\alpha \beta)^{\text{th}}$ entry of $\Lambda_k(I)=\prod_{j=1}^k u_{i_j}$, where $I = \left(i_1,i_2,\ldots,i_k\right)$. Take $\gamma=\min\{\beta-\alpha,k\}$. Now we will be finding all possible monomials contributing to the entry corresponding the smaller upper triangular matrices of size $2\le\epsilon\le \gamma+1$. 

\noindent \textbf{Step 1:} Set $\epsilon=2$, then the $2\times 2$ upper triangular matrix coming from $u_i$ is given by 
\begin{displaymath}
    \begin{pmatrix}
        a_{\alpha \alpha}(i) & a_{\alpha \beta}(i) \\
                             & a_{\beta \beta}(i)
    \end{pmatrix}.
\end{displaymath}
The monomial corresponding to $\epsilon=2$ will be given by $P_{\alpha \beta}^{(1)} \left(\Lambda_k(I)\right)$
\begin{equation}\label{eq:step-1}
    \sum_{1\le s_1\le k} \left(\prod_{j=1}^{s_1-1}a_{\alpha \alpha}(i_j)\right)a_{\alpha \beta}(i_{s_1})\left(\prod_{j=s_1+1}^{k}a_{\beta \beta}(i_j)\right).
\end{equation} 

\noindent \textbf{Step 2:}  Set $\epsilon=3$, then the $3\times 3$ upper triangular matrices from $u_i$ is given by
\begin{displaymath}
    \begin{pmatrix}
        a_{\alpha \alpha}(i) & a_{\alpha \eta}(i) & a_{\alpha \beta}(i) \\
        & a_{\eta \eta}(i) & a_{\eta \beta}(i)\\
        &                  & a_{\beta \beta}(i)
    \end{pmatrix}.
\end{displaymath}
Choose $\eta$ such that $\alpha < \eta < \beta$, then the monomial corresponding to $\epsilon=3$  is given by $P_{\alpha \beta}^{(2)}\left(\Lambda_k(I)\right)$ 
\begin{equation} \label{eq:step-2}
    \sum_{\alpha< \eta < \beta}\left[\sum_{1\le s_1<s_2\le k}\left(\prod_{j=1}^{s_1-1}a_{\alpha \alpha}(i_j)\right) a_{\alpha \eta}(i_{s_1}) \left(\prod_{j=s_1+1}^{s_2-1}a_{\eta \eta}(i_j)\right)a_{\eta \beta}(i_{s_2})  \left(\prod_{j=s_2+1}^{k}a_{\beta \beta }(i_j)\right)\right].
\end{equation}

\noindent \textbf{Step r:}  Set $\epsilon=r+1$, then the $(r+1)\times(r+1)$ upper triangular matrices from $u_i$ is given by
\begin{displaymath}
    \begin{pmatrix}
        a_{\alpha \alpha}(i) & a_{\alpha \eta_1}(i) & \cdots & a_{\alpha\eta_r}(i) & a_{\alpha \beta}(i) \\
        & a_{\eta_1 \eta_1}(i) & \cdots & a_{\eta_1\eta_r}(i) & a_{\eta_1 \beta}(i)\\
        &  & \ddots & \vdots & \vdots\\
        &    &           &   & a_{\beta \beta}(i)
    \end{pmatrix}.
\end{displaymath}
Choose $\eta_1,\eta_2,\ldots,\eta_r$ such that $\alpha < \eta_1 < \eta_2 < \cdots <\eta_r < \beta$, then the monomial corresponding to $\epsilon=r+2$  is given by $P_{\alpha \beta}^{(r)}\left(\Lambda_k(I)\right)$
\begin{equation} \label{eq:step-r}
    \begin{gathered}
        \sum_{\alpha< \eta_1 < \cdots <\eta_r < \beta}\left[\sum_{1\le s_1<\cdots<s_{r+1}\le k}
        \left(\prod_{j=1}^{s_1-1}a_{\alpha \alpha}(i_j)\right) a_{\alpha \eta_1}(i_{s_1}) \left(\prod_{j=s_1+1}^{s_2-1}a_{\eta_1 \eta_1}(i_j)\right)
        a_{\eta_1 \eta_2}(i_{s_2})\times \cdots\right.
        \\
        \left.\times \left(\prod_{j=s_{r}+1}^{s_{r+1}-1}a_{\eta_{r} \eta_{r}}(i_j)\right)
        a_{\eta_{r} \beta}(i_{s_{r+1}})  \left(\prod_{j=s_{r+1}+1}^{k}a_{\beta \beta }(i_j)\right)\right].
    \end{gathered}
\end{equation}

\noindent Since $\gamma$ is finite, we have finitely many steps. Now the final step is given as follows;

\noindent \textbf{Step $\gamma+2$:}  Add all the monomials together to obtain the entry $\lambda_{\alpha\beta}$. 
\begin{example}
To have a better understanding of the algorithm we present a diagrammatic view of the monomials. Choose $n=6$, $\alpha=1$ and $\beta=6$. Then to compute the $16$-th entry of the product $\prod_{i=1}^{7}u_i$ we proceed as follows. 
\begin{figure}[H]
    \centering
    \includegraphics[width=0.22\textwidth]{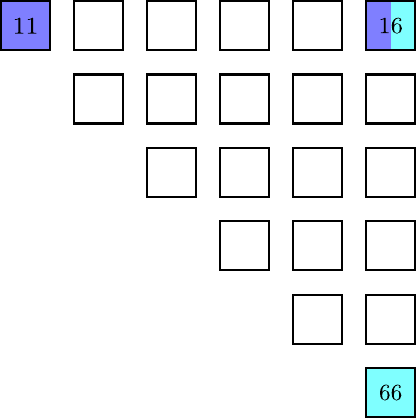}
    \caption{Combinations for $2$-fold product\label{fig:prodTwo}}
\end{figure}
\begin{figure}[H]
    \centering
    \includegraphics[width=\textwidth]{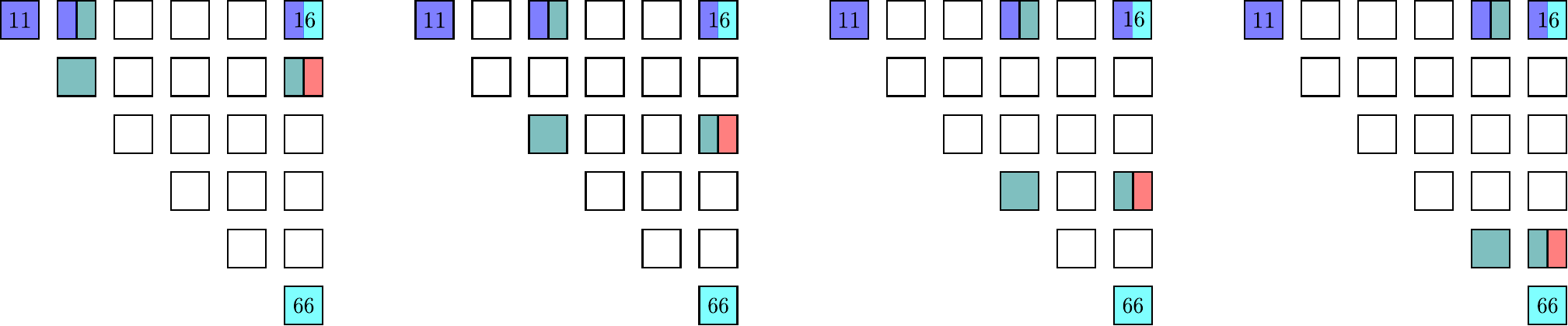}
    \caption{Combinations for $3$-fold product\label{fig:prodThree}}
\end{figure}

\begin{figure}[H]
    \centering
    \includegraphics[width=\textwidth]{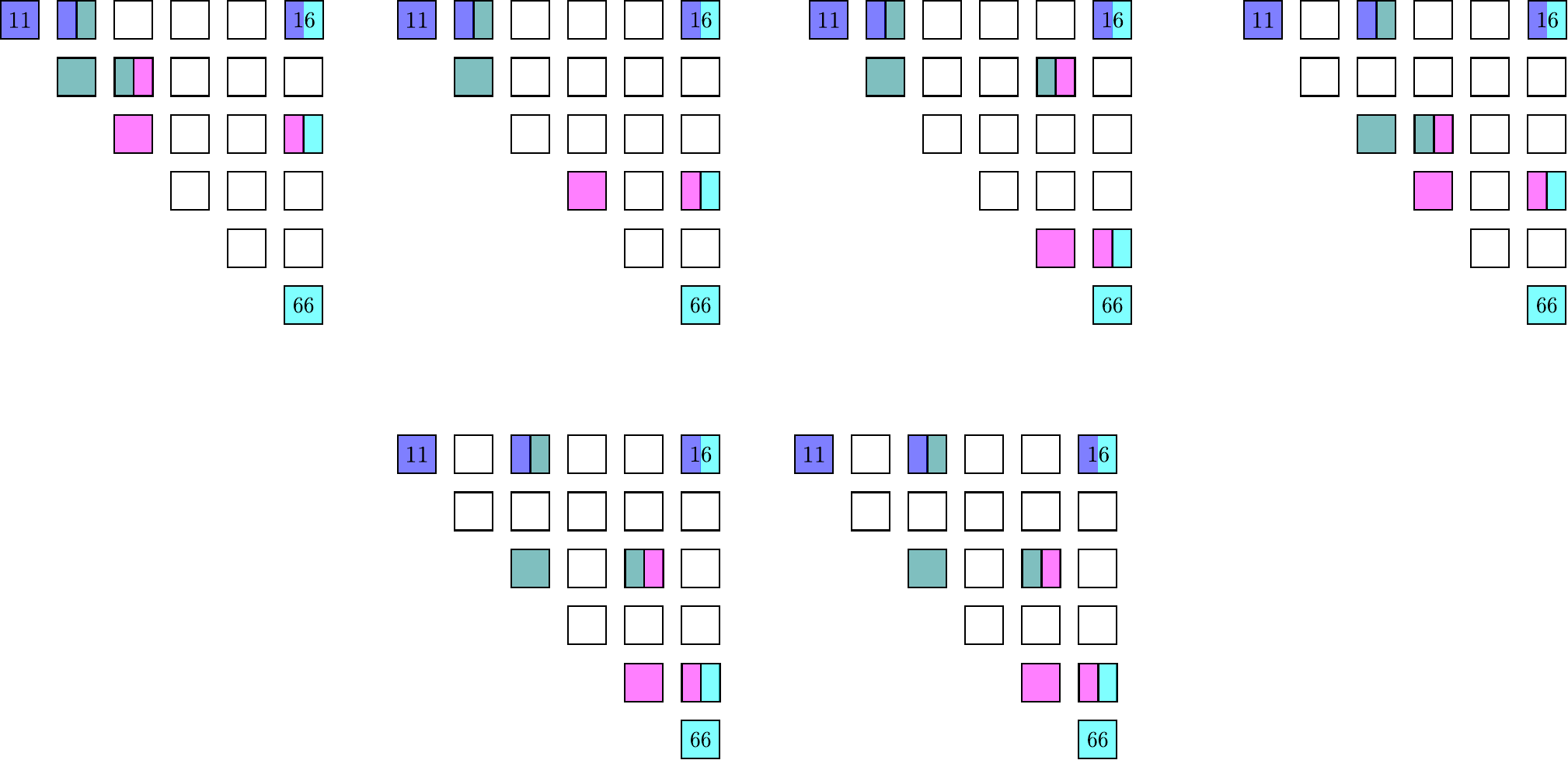}
    \caption{Combinations for $4$-fold product\label{fig:prodFour}}
\end{figure}

\begin{figure}[H]
    \centering
    \includegraphics[width=\textwidth]{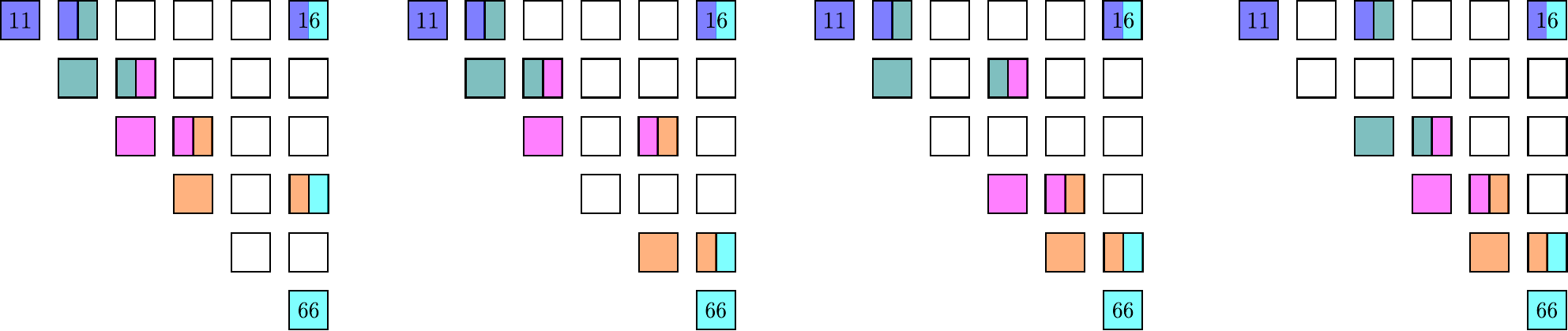}
    \caption{Combinations for $5$-fold product\label{fig:prodFive}}
\end{figure}

\begin{figure}[H]
    \centering
    \includegraphics[width=0.22\textwidth]{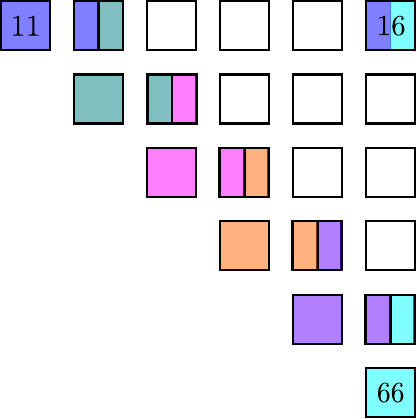}
    \caption{Combinations for $6$-fold product\label{fig:prodSix}}
\end{figure}
\noindent Note that there will not be any term involving $7$-fold product since $\gamma=5$.
\end{example}
Now we present the following proposition.
\begin{proposition}\label{prop:generalizedMultiplication}
    Let $u_i\in T_n(R)$ be matrices with $(u_i)_{uv}=a_{uv}(i)$, for $i=1,2,\ldots,k$. If $\lambda_{\alpha\beta}$ denotes the 
    $\alpha\beta^{\text{th}}$ entry of the product $\prod\limits_{j=1}^ku_{l_j}$, then 
    \begin{align*}
        \lambda_{\alpha\beta}=\sum\limits_{r=1}^{\beta-\alpha}P_{\alpha \beta}^{(r)} \left(\prod\limits_{j=1}^ku_{l_j}\right).
    \end{align*}
\end{proposition}
\begin{proof}
    As this can be proved by induction on $k$ as in the case $n=4$, this is left to the reader.
\end{proof}

In this case $n=4$, each coordinate is determined by a family of multi-index polynomials. These polynomials $\{p_{\Vec{I}}\}$ solely
depend on the polynomial we started with, namely $p$. As evident from the general formula of the multiplication,
the general formula for an entry in the position with higher indices depends on smaller indices as shown in the following diagram.

\begin{figure}[H]
    \centering
    \includegraphics[width=0.3\textwidth]{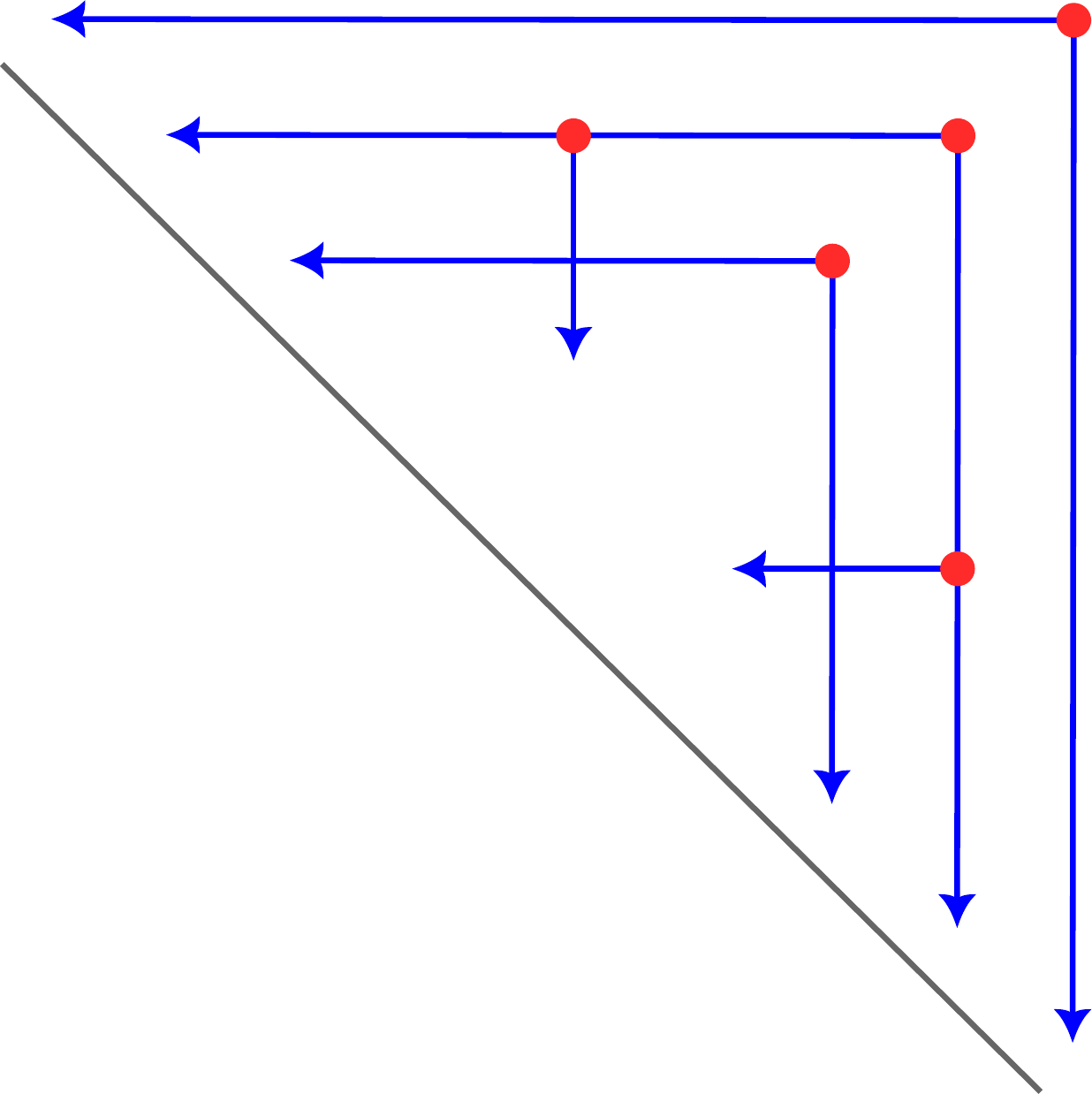}
    \caption{Submatrix\label{fig:submatrix}}
\end{figure}

This is true in general as evident from the following proposition.

\begin{lemma}
    For a polynomial $f\in R \left[F_n\right]$ let $P_{\alpha \beta}^{(r)}(f)$ denote the $\alpha \beta^{\text{th}}$ entry of $f \left(\vec{\mathbf{u}}\right)$  of weight $r$. Then the following statements are true for two monomials $g,h\in F_n$.
    \begin{enumerate}[(i)]
        \item $P_{\alpha \beta}^{(r)}(g+h) = P_{\alpha \beta}^{(r)}(g) + P_{\alpha \beta}^{(r)}(h)$.
        \item For any $\delta\in K$, $P_{\alpha \beta}^{(r)}(\delta g) = \delta P_{\alpha \beta}^{(r)}(g)$. 
    \end{enumerate}    
\end{lemma}
\begin{proof}
    Note that for any two matrices $U,V\in \tk[m]$, $c\in K$ we have \begin{align*}
        & [U+V]_{\alpha \beta} = U_{\alpha \beta} + V_{\alpha \beta}, \\
        & [c U]_{\alpha \beta} = c U_{\alpha \beta}.
    \end{align*}
\end{proof}
\noindent The following proposition follows immediately using the above lemma.
\begin{proposition}\label{prop:polyEntry}
    Let $p\in R[F_n]$, $u_i\in T_m(R)$ for $1\leq i\le n$, then the $\alpha \beta^{\text{th}}$ entry of $p(\vec{\mathbf{u}})$ is given by the following expression:
    \begin{equation}\label{eq:polynomialUnderProd}
        \left[p \left(\vec{\mathbf{u}}\right)\right]_{\alpha \beta} = \sum_{r=1}^{\beta-\alpha} P_{\alpha \beta}^{(r)}(p).
    \end{equation}
    Here $P_{\alpha \beta}^{(r)}(p)$ is defined as follows. Suppose 
    \begin{displaymath}
        p = \sum_I \mu_Ix_{i_1}x_{i_2}\cdots x_{i_k},\quad I = \left(i_1,i_2,\ldots,i_k\right),
    \end{displaymath}
    then
    \begin{displaymath}
        P_{\alpha \beta}^{(r)}(p) = \sum_I \mu_I P_{\alpha \beta}^{(r)} \left(\Lambda_k(I)\right).
    \end{displaymath} 
\end{proposition}
\begin{proof}
    Note that using \autoref{prop:generalizedMultiplication}, we have
    \begin{align*}
        \left[p \left(\vec{\mathbf{u}}\right)\right]_{\alpha \beta} & = \sum_{r=1}^{\beta-\alpha} P_{\alpha \beta}^{(r)}(p) = \sum_{r=1}^{\beta-\alpha} \sum_I \mu_I P_{\alpha \beta}^{(r)} \left(\Lambda_k(I)\right).
    \end{align*}
\end{proof}

\begin{defn}
    Given a polynomial $p$, the polynomials $P_{\alpha \beta}^{(r)} \left( p \right)$ in \autoref{prop:polyEntry} will be called as \emph{\mip}.
\end{defn}
Here we present two examples of \mip s.
\begin{example}
    Let $K$ be an algebraically closed field and  consider the following polynomial $p(x,y,z) = x^2+y+z\in K[x,y,z]$. Let 
    \begin{equation}\label{eq:ui}
        u_i = 
        \begin{pmatrix}
            a_i & b_i & c_i & d_i
            \\
                        & e_i & f_i & g_i
            \\
                        &           & h_i & j_i
            \\
                        &           &            & l_i
        \end{pmatrix}\in T_4(K)   
    \end{equation} 
    for $i=1,2,3$. 
    We will compute $p \left(u_1,u_2,u_3\right)$. Then $p \left(u_1,u_2,u_3\right)$ will be
    \begin{align*}
        \begin{pmatrix}
            a_1^2+a_2+a_3
            & \left(a_1+e_1\right)b_1 + b_2 + b_3
            & \left(a_1 + h_1 \right)c_1 + c_2 + c_3 
            & \left(a_1 + h_1 \right)c_1 + c_2 + c_3 
            \\
            & 
            &  + b_1 f_1 
            &  + b_1 g_1 + c_1 j_1 
            \\[2ex]
            & e_1^2+e_2+e_3
            & \left(e_1+h_1\right)f_1 + f_2 + f_3
            & \left(e_1 + l_1 \right)g_1 + g_2 + g_3 
            \\
            &
            & 
            & + f_1 j_1 
            \\[2ex]
            &   
            & h_1^2+h_2+h_3
            & \left(h_1+l_1\right)j_1 + j_2 + j_3
            \\[2ex]
            & 
            & 
            &  l_1^2+l_2+l_3
        \end{pmatrix}
    \end{align*}
    Here the \mip s are given by
    \begin{displaymath}
        p_i \left(x_1, x_2, x_3, y_1, y_2, y_3 \right) = 
        \begin{cases}
            x_1 + y_1, & \text{ if } i = 1\\
            1,         & \text{otherwise}
        \end{cases};
    \end{displaymath}
    \begin{displaymath}
        p_{ij} \left(x_1, x_2, x_3, y_1, y_2, y_3,z_1,z_2,z_3 \right) = 
        \begin{cases}
            1, & \text{ if } i = j = 1\\
            0,         & \text{otherwise}
        \end{cases};
    \end{displaymath} 
    and 
    \begin{displaymath}
        p_{ijk} \equiv 0 \text{ for } i,j,k=1,2,3.
    \end{displaymath} 
\end{example}

\begin{example}
    Take 
    \begin{displaymath}
        p(x,y) = xy + yx
    \end{displaymath} 
    For the same $u_i$ taken in \autoref{eq:ui}, we have $p \left(u_1, u_2\right)$ 
    \begin{align*}
        \begin{pmatrix}
                2a_1 a_2 
            & \left(a_2 + e_2\right)b_1 
            & \left(a_2 + h_2 \right)c_1
            & \left(a_2 + l_2 \right)d_1
            \\
            & + \left(a_1 + e_1 \right)b_2
            & + \left(a_1 + h_1 \right)c_2
            & + \left(a_1 + l_1 \right)c_2
            \\
            & 
            & + b_1 f_2 + b_2f_1
            & + b_1 g_2 + b_2g_1 + c_1j_2 + c_2j_1
            \\[2ex]
            & 2e_1e_2 
            & \left(e_2 + h_2\right)f_1 
            & \left(e_2 + l_2 \right)g_1
            \\
            &
            & + \left(e_1 + h_1\right)f_2
            & + \left(e_1 + l_1\right)g_2
            \\
            &
            & 
            & + f_1j_2 + f_2 j_1 + 
            \\[2ex]
            &   
            & 2h_1 h_2 
            & \left(h_2 + l_2\right)j_1 
            \\
            &   
            & 
            & \left(h_1 + l_1\right)j_2
            \\[2ex]
            & 
            & 
            & 2 l_1 l_2 
        \end{pmatrix}
    \end{align*}
    Here we have
    \begin{displaymath}
        p_i \left(x_1, x_2, x_3, y_1, y_2, y_3 \right) = 
        \begin{cases}
            x_2 + y_2, & \text{ if } i = 1\\
            x_1 + y_1, & \text{ if } i = 2
        \end{cases};
    \end{displaymath}
    \begin{displaymath}
        p_{ij} \left(x_1, x_2, x_3, y_1, y_2, y_3,z_1,z_2,z_3 \right) = 
        \begin{cases}
            1, & \text{ if } i = 1,~j = 2\\
            1, & \text{ if } i = 2,~j = 1\\
            0,         & \text{otherwise}
        \end{cases};
    \end{displaymath} 
    and 
    \begin{displaymath}
        p_{ijk} \equiv 0 \text{ for } i,j,k=1,2,3.
    \end{displaymath} 
\end{example}

Next we present a generalization of \autoref{lem:polyNonZero} which was a crucial ingredient to prove \autoref{thm:mainThmFor_n=4}.
\begin{proposition}\label{prop:ifOneNonzeroThenManynonzero}
    Let $k\ge 2$, $f\in K \left[F_{kn}\right]$. Suppose there exists $\avect{1},\ldots,\avect{n}\in K^k$ such that 
    \begin{displaymath}
        f \left(\avect{1},\ldots,\avect{n}\right) \neq 0.
    \end{displaymath}  
    Then for any $l\ge 1$ there exist vectors $\avect{1}',\ldots,\avect{n}',\avect{n+1}',\ldots,\avect{n+l}'\in K^k$ satisfying
    \begin{displaymath}
        f \left(\avect{\sigma(1)}',\ldots,\avect{\sigma(n)}'\right) \neq 0,
    \end{displaymath} 
    for all $\sigma\in S_{n+l}$ where $S_{n+l}$ is the symmetric group on $(n+l)$ letters.
\end{proposition}

\begin{proof}
    The proof is similar to the proof of \autoref{lem:polyNonZero}, and hence we leave it to the reader.
\end{proof}

\section{Main results}\label{sec:mainResults}

\subsection{Image of polynomials on \texorpdfstring{$\tk[m]$}{TmK}}\label{sec:mainThm}

In the section we will generalize the results obtained in \autoref{sec:n=4}. In \autoref{lem:ZariskiDense} we will show that image of order zero polynomials are Zariski dense. \autoref{thm:order_m-1}, \autoref{thm:order_1}, and \autoref{thm:order_t} are concerned about images of polynomials of order between $1$ and $m-1$. Let us start by the following result which will be used to conclude \autoref{lem:ZariskiDense}.
\begin{lemma}\label{lem:equivalence}
    Inside $\tk[m]$ we have the following similarity of matrices.
    \begin{displaymath}
        \begin{pmatrix}
            a_{11} & a_{12} & \cdots & a_{1m} \\ 
                   & a_{22} & \cdots & a_{2m} \\ 
                   &        & \ddots & \vdots \\ 
                   &        &        & a_{mm} \\ 
        \end{pmatrix} \sim
        \diag \left(a_{11},a_{22},\ldots,a_{mm}\right)
    \end{displaymath}
    when $a_{ii}\neq a_{jj}$ for $i\neq j$. 
\end{lemma}
\begin{proof}
    Let $i,j$ be the position such that $a_{ij}\neq 0$ and $j-i$ is minimal. Then an appropriate $t$ the matrix $\left(I+t E_{ij}\right)A \left(I+t E_{ij}\right)^{-1}$ has $ij^{\text{th}}$ entry to be zero, where $E_{ij}$ denotes the elementary matrix with $ij^{\text{th}}$-entry non-zero. This process can be further continued to obtain the desired result.
\end{proof}

We can give the immediate result using the above lemma.
\begin{lemma}\label{lem:diagonal}
    Let $p\in K[F_n]$ satisfying $p(K)\neq 0$. Then we have that
    \begin{align*}
        \left\{A=(a_{\alpha\beta})\in T_m(K):a_{\delta\delta}\neq a_{\sigma\sigma}\text{ for all }1\leq\delta\neq\sigma\leq m\right\}\subseteq p\left(T_m(K)\right).
    \end{align*}
\end{lemma}
\begin{proof}
    Let $A$ satisfy that all the diagonal entries are pairwise distinct. Then the following equivalence follows from \autoref{lem:equivalence}. 
    
    \begin{align*}
        \begin{pmatrix}
            a_{11} & a_{12} & \cdots & a_{1m} \\ 
                   & a_{22} & \cdots & a_{2m} \\ 
                   &        & \ddots & \vdots \\ 
                   &        &        & a_{mm} \\ 
        \end{pmatrix} \sim
        \diag \left(a_{11},a_{22},\ldots,a_{mm}\right).
    \end{align*}

    Now note that, if $u_i\in T_m(K)$, $v\in T_m(K)^\times$, $p\in K[F_n]$, we have that
    \begin{align*}
        v\cdot p(u_1,u_2,\ldots,u_n)\cdot v^{-1}=p(vu_1v^{-1},vu_2v^{-1},\ldots,vu_nv^{-1}).
    \end{align*}
    Indeed, $\prod\limits_{j=1}^l(vx_{i_j}v^{-1})=vx_{i_1}v^{-1}\cdot vx_{i_2}v^{-1}\cdots vx_{i_{l-1}}v^{-1}vx_{i_l}v^{-1}$. Hence, to prove the containment it is enough to show that 
    \begin{align*}
        \textbf{diag}(a_{11},a_{22},\ldots, a_{mm})\in p\left(T_m(K)\right).
    \end{align*}
    
    Since $p(K)\neq 0$, for $1\leq \alpha\leq m$ and $1\leq j\leq n$, there exists $a_{\alpha\alpha}(j)\in K$, such that $p(a_{ii}(1),a_{ii}(2),\ldots,a_{ii}(n))=a_{ii}$, by \cite[Lemma 3.1]{WaZhLu21}. Hence, considering \begin{align*}
        u_i=\textbf{diag}(a_{11}(i),a_{22}(i),\ldots, a_{mm}(i)),
    \end{align*}
    we get that $\textbf{diag}(a_{11},a_{22},\ldots, a_{mm})\in p(T_m(K))$.
    This finishes the proof.
\end{proof}

\begin{proposition}\label{lem:ZariskiDense}
    Let $p\in K[F_n]$ satisfy $p(K)\neq 0$. Then $p(T_m(K))$ is a Zariski dense subset of $T_m(K)$.
\end{proposition}
\begin{proof}
    This has been proved in case of $n=2,3$ in \cite{WaZhLu21,WaZhLu22}. We present a proof for the general case, necessarily imitating the proof therein. Consider $T_m(K)$ as the 
    affine space $K^{\frac{m(m+1)}{2}}$.
    By \autoref{lem:diagonal}, we have that
    \begin{align*}
        \D=\left\{A=(a_{\alpha\beta})\in T_m(K):a_{\delta\delta}\neq a_{\sigma\sigma}\text{ for all }1\leq\delta\neq\sigma\leq m\right\}\subseteq p(T_m(K)).
    \end{align*}
    Hence, it is enough to prove that $\D$ is dense in $T_m(K)$. 

    Take a point $P=\left(a_{11}, a_{22}, \ldots, a_{mm},a_{12}, a_{23},\ldots,a_{m-1m},\ldots, a_{1m}\right)\in K^{\frac{m(m+1)}{2}}$ and an open set $U$ containing $P$. Hence, there exists an ideal $I$ of $K\left[y_1,\ldots,y_{\frac{m(m+1)}{2} }\right]$ such that $U=K^{\frac{m(m+1)}{2}}\setminus Z(I)$, where $Z(I)$ denotes the zero set of $I$. This implies there exists $f\in I$ with $f(P)\neq 0$ that is
    \begin{displaymath}
        f \left(a_{11}, a_{22}, \ldots, a_{mm},a_{12}, a_{23},\ldots,a_{m-1m},\ldots, a_{1m}\right)\neq 0.
    \end{displaymath} 
    If $a_{ii} \neq a_{j j}$ for all $i\neq j$, then $P\in \mathcal{D}$. So we will assume that there exists some $i_0\neq j_0$ but $a_{i_0i_0}=a_{j_0j_0}$. Without loss of generality assume that $i_0=1$ and $j_0=2$. Set $a_{11}' = a_{11}$. We want to find $a_{22}'\neq a_{22}$ such that $f \left(a_{11}', a_{22}', \ldots, a_{mm},a_{12}, a_{23},\ldots,a_{m-1m},\ldots, a_{1m}\right)\neq 0$. Define 
    \begin{displaymath}
        h_{22}(x) = f \left(a_{11}', x, \ldots, a_{mm},a_{12}, a_{23},\ldots,a_{m-1m},\ldots, a_{1m}\right).
    \end{displaymath} 
    Note that the function is not identically zero. As $K$ is an algebraically closed field (hence infinite), there exists $a_{22}'\notin \left\{a_{11}'\right\}$  such that $h_{22}\left(a_{22}'\right)\neq 0$. 

    Since $K$ is infinite, this procedure can be continued to obtain an element 
    \begin{displaymath}
        P'=\left(a_{11}',a_{22}',\ldots,a_{mm}',a_{12},\ldots,a_{1m}\right)
    \end{displaymath} 
    such that $a_{ii}'\neq a_{j j}'$ for all $i\neq j$  and $f \left(P'\right)\neq 0$. Hence, $P'\in \mathcal{D}$.
\end{proof}

\begin{proposition}\label{thm:order_m-1}
    Let $p\in K[F_n]$ be a polynomial in non-commutative variable. Suppose $\ord=m-1$. Then $p(T_m(K))=T_m(K)^{(m-2)}$.
\end{proposition}
\begin{proof}
    Since $\ord=m-1$, we have $p \left(T_{m-1}(K)\right)=0$ but $p \left(T_m(K)\right)\neq 0$. Since $p \left(T_{m-1}(K)\right)=0$, the \mip s of index size less than or equal to $m-2$ are identically zero. Hence, $p \left(T_m(K)\right)\subseteq T_m(K)^{(m-2)}$. For showing the other inequality, we take a matrix $X\in T_n(K)^{(m-2)}$ such that 
    \begin{displaymath}
        X_{ij} = 
        \begin{cases}
            x, & \text{ if } i=1, j =n\\
            0,      & \text{otherwise}.
        \end{cases}.
    \end{displaymath} 
    Since $p \left(T_m(K)\right)\neq 0$, there exists a multi-index $I$  of size $m-1$ such that $p_I$ is not identically zero, say $I_0=\left(i_1,i_2,\ldots,i_{m-1}\right)$. Let $p_I \left(\avect{11},\avect{22},\ldots,\avect{mm}\right)\neq 0$ for some $\avect{11},\ldots,\avect{mm}\in K^{n}$. Set
    \begin{displaymath}
        \lambda_0 = p_{I_0} \left(\avect{11},\avect{22},\ldots,\avect{mm}\right).
    \end{displaymath} 
    At first assume that $i_1=i_2=\cdots=i_{m-1}$. Set
    \begin{displaymath}
        u_i = 
        \begin{cases}
            \begin{pmatrix}
                a_{11}(i) & 0 & 0 & \cdots & 0 \\
                          & a_{22}(i) & 0 & \cdots & 0 \\
                          &           & \ddots & \cdots & \vdots\\
                          &           & & & a_{mm}(i)
            \end{pmatrix}, &  \text{ if } i\neq i_1 \\[1ex]
            \begin{pmatrix}
                a_{11}(i) & \lambda_0^{-1}x & 0 & \cdots & 0 \\
                          & a_{22} \left(i_1\right) & 1 & \cdots & 0 \\
                          &           & \ddots & \cdots & \vdots\\
                          &           & & a_{m-1m-1}\left(i_{1}\right) & 1\\
                          &           & & & a_{mm}\left(i_{1}\right)
            \end{pmatrix}, &  \text{ if } i = i_1
        \end{cases}.
    \end{displaymath} 
    Then $p \left(\vec{\mathbf{u}}\right)=X$ and hence $p(T_m(K))=T_m(K)^{(m-2)}$.\\

    We now assume that $i_1\neq i_2=i_3=\cdots=i_{m-1}$. In that case we take 
    \begin{displaymath}
        u_i = 
        \begin{cases}
            \begin{pmatrix}
                a_{11}(i) & 0 & 0 & \cdots & 0 \\
                          & a_{22}(i) & 0 & \cdots & 0 \\
                          &           & \ddots & \cdots & \vdots\\
                          &           & & & a_{mm}(i)
            \end{pmatrix}, &  \text{ if } i\neq i_1,i_2 \\[1ex]
            \begin{pmatrix}
                a_{11}(i) & \lambda_0^{-1}x & 0 & \cdots & 0 \\
                          & a_{22} \left(i_1\right) & 0 & \cdots & 0 \\
                          &           & \ddots & \cdots & \vdots\\
                          &           & & & a_{mm}\left(i_{1}\right)
            \end{pmatrix}, &  \text{ if } i = i_1 
            \\[1ex]
            \begin{pmatrix}
                a_{11}(i) & 0 & 0 & \cdots & 0 \\
                          & a_{22} \left(i_2\right) & 1 & \cdots & 0 \\
                          &           & \ddots & \cdots & \vdots\\
                          &           & & a_{m-1 m-1}\left(i_2\right) & 1 \\
                          &           & & & a_{mm}\left(i_{2}\right)
            \end{pmatrix}, &  \text{ if } i = i_2
        \end{cases}.
    \end{displaymath} 
    Again note that $p(\vec{\mathbf{u}})=X$  and hence we are done. In this way, it can be shown that there exists $\vec{\mathbf{u}}$ such that $p \left(\vec{\mathbf{u}}\right)=X$.
\end{proof}

\begin{proposition}\label{thm:order_1}
    Let $p\in K[F_n]$ be a polynomial in non-commutative variable. Suppose $\ord=1$. Then $p(T_m(K)) = T_m(K)^{(0)}$.
\end{proposition}
\begin{proof}
    Since $\ord=1$, we have $p \left(T_{1}(K)\right)=0$ but $p \left(T_{2}(K)\right)\neq 0$. Since $p \left(T_{1}(K)\right)=0$, the \mip s of index size less than or equal to $1$ are identically zero. Hence, $p \left(T_m(K)\right)\subseteq T_{m}(K)^{(0)}$. 

    \noindent For the reverse inclusion, let 
    \begin{displaymath}
        X = 
        \begin{pmatrix}
            0 & x_{12} & x_{13} & \cdots & x_{1m} \\
              &   0    & x_{23} & \cdots & x_{2m} \\
              &        & \vdots & \ddots & \vdots \\
              &        &   0    &   & x_{mm} \\
        \end{pmatrix}.
    \end{displaymath} 
    We want to find $\vec{\mathbf{u}}= \left(u_1,\ldots,u_n\right)\in \tk[m]^n$ such that $p \left(\vec{\mathbf{u}}\right)=X$. We will give an algorithm to find $\vec{\mathbf{u}}$. First, observe that $p \left(\tk[2]\right)\neq 0$, so there exists $i_0\in \{1,2\ldots,n\}$ such that 
    \begin{displaymath}
        p_{i_0} \left(\avec{1,m}\right)\neq 0.
    \end{displaymath} 
    Using \autoref{prop:ifOneNonzeroThenManynonzero}, we can find $\avect{11}',\ldots,\avect{mm}'\in K^n$ for all $\sigma\in S_m$ such that 
    \begin{displaymath}
        \lambda_{\alpha\beta}=p_{i_0} \left(\avec{\alpha,\beta}'\right) \neq 0.
    \end{displaymath} 
    Now we define $u_i$ to be $\diag \left(a_{11}(i),\ldots,a_{mm}(i)\right)$ if $i\neq i_0$. If $i=i_0$, we take $u_{i_0}$ as follows.
    \begin{itemize}
        \item The $(ll)$-entry will be $\left(u_{i_0}\right)_{ll}=a_{ll}\left(i_0\right)$.
        \item The $(ll+1)$-entry will be $\left(u_{i_0}\right)_{ll+1}=\lambda_{l l+1}^{-1}x_{l l+1}$.
        \item The $(ll+2)$-entry will be 
        \begin{displaymath}
            \left(u_{i_0}\right)_{ll+2}=\lambda_{ll+2}^{-1} \left(x_{ll+2}-\lambda_{ll+1l+2}\left(u_{i_0}\right)_{ll+1}\left(u_{i_0}\right)_{l+1l+2}\right)
        \end{displaymath} 
        \item The $(ll+3)$-entry will be 
        \begin{align*}
            \left(u_{i_0}\right)_{ll+3} & =\lambda_{ll+3}^{-1} \left(x_{ll+3}-\lambda_{ll+1l+3}\left(u_{i_0}\right)_{ll+1}\left(u_{i_0}\right)_{l+1l+2}\right. \\
            & \left.-\lambda_{ll+2l+3}\left(u_{i_0}\right)_{ll+2}\left(u_{i_0}\right)_{l+2l+3}-\lambda_{ll+1l+2l+3}\left(u_{i_0}\right)_{ll+1}\left(u_{i_0}\right)_{l+1l+2}\left(u_{i_0}\right)_{l+2l+3}\right).
        \end{align*} 
        \item Similarly, one can write the $(ll+k)$-entry in terms of $\left(u_{i_0}\right)_{ll+\alpha}$ where $\alpha=1,2,\ldots,k-1$.   
    \end{itemize}
    It can be checked that $p \left(\vec{\mathbf{u}}\right)=X$. Hence, we have $p \left(\tk[m]\right)=\tk[m]^{(0)}$.
\end{proof}

\begin{proposition}\label{thm:order_t}
    Let $p\in K[F_n]$ be a polynomial in non-commutative variable. Suppose $1<\ord = t <m-1$. Then $p(T_m(K)) \subseteq T_m(K)^{(t-1)}$.
\end{proposition}

\begin{proof}
    Since $\ord=t$, we have $p \left(T_{t}(K)\right)=0$ but $p \left(T_{t+1}(K)\right)\neq 0$. Since $p \left(T_{t}(K)\right)=0$, the \mip s of index size less than or equal to $t-1$ are identically zero. Hence, $p \left(T_m(K)\right)\subseteq T_{m}(K)^{(t-1)}$. 
\end{proof}

\noindent Note that in the above proposition, the containment can be strict. Here is a supporting example.

\begin{example}\label{eg:counterExample}
    Let $m=5,n=2$ and $p(x_1,x_2)=\left(x_1x_2-x_2x_1\right)^2$. Note that $\ord=t=2$ (see \autoref{eg:order2}). We show that $p\left(T_5(K)\right) \neq T_5(K)^{(1)}$. Let  
    \begin{align*}
        X_1 = 
        \begin{pmatrix}
            a_1 & b_1 & c_1 & d_1 & e_1 \\
            0 & f_1 & g_1 & h_1 & i_1 \\
            0 & 0 & j_1 & k_1 & l_1 \\
            0 & 0 & 0 & m_1 & n_1 \\
            0 & 0 & 0 & 0 & o_1 
        \end{pmatrix} \text{ and } 
        X_2 = 
        \begin{pmatrix}
            a_2 & b_2 & c_2 & d_2 & e_2 \\
            0 & f_2 & g_2 & h_2 & i_2 \\
            0 & 0 & j_2 & k_2 & l_2 \\
            0 & 0 & 0 & m_2 & n_2 \\
            0 & 0 & 0 & 0 & o_2 
        \end{pmatrix}.
    \end{align*}
    Then, 
    \begin{align*}
        p(X_1,X_2)=\left(
            \begin{array}{ccccc} 
                0 
                & 0 
                & \substack{\left(f_{1}\,g_{2}-f_{2}\,g_{1}+g_{1}\,j_{2}-g_{2}\,j_{1}\right)\\ \left(a_{1}\,b_{2}-a_{2}\,b_{1}+b_{1}\,f_{2}-b_{2}\,f_{1}\right)} 
                & \star
                & \star
                \\[1ex]
                  0 
                & 0 
                & 0 
                & \substack{\left(f_{1}\,g_{2}-f_{2}\,g_{1}+g_{1}\,j_{2}-g_{2}\,j_{1}\right)\\ \left(j_{1}\,k_{2}-j_{2}\,k_{1}+k_{1}\,m_{2}-k_{2}\,m_{1}\right) }
                & \star
                \\[1ex]
                0 
                & 0 
                & 0 
                & 0 
                & \substack{\left(j_{1}\,k_{2}-j_{2}\,k_{1}+k_{1}\,m_{2}-k_{2}\,m_{1}\right)\\ \left(m_{1}\,n_{2}-m_{2}\,n_{1}+n_{1}\,o_{2}-n_{2}\,o_{1}\right)}
                \\ 
                0 & 0 & 0 & 0 & 0
                \\ 
                0 & 0 & 0 & 0 & 0 
            \end{array}
        \right)
    \end{align*}
    Hence, it is easy to see that the matrix
    \begin{displaymath}
        A = 
        \begin{pmatrix}
            0 & 0 & 1 & 0 & 0 \\
            0 & 0 & 0 & 0 & 0 \\
            0 & 0 & 0 & 0 & 1 \\
            0 & 0 & 0 & 0 & 0 \\
            0 & 0 & 0 & 0 & 0 \\
        \end{pmatrix}
    \end{displaymath}
    is not in the image of $p$. 
\end{example}

\hf The natural question in this regard is that \emph{can every element of $T_m(K)^{(t-1)}$ be written as sum of $d$ many elements from $p\left( T_m(K) \right) $ for some $d\geq 2$}? These are known as Waring type problem. We prove existence of such $d$ in the next subsection. 
\subsection{Waring type problem}\label{sec:Waring}

There has been a growing interest about the Waring problem for matrix algebras. For example, it has been proved that any traceless matrix can be written as sum of two matrices from $f\left( M_n(\mathbb{C} ) \right) -  f\left( M_n(\mathbb{C} ) \right)$, where $f$ is neither an identity nor a central polynomial for $M_n(\mathbb{C} )$ (see \cite{BrVs22} by Bre{\v s}ar and {\v S}emrl). Recently, they have also proved that if $\alpha _1,\alpha _2,\alpha _3\in \mathbb{C} \setminus \{0\}$ and $\alpha _1+\alpha _2+\alpha _3=0$, then any traceless matrix over $\mathbb{C} $ can be written as $\alpha _1A_1 + \alpha _2A_2+\alpha _3A_3$, where $A_i\in f\left( M_n(\mathbb{C} ) \right) $ (see \cite{BrVs_unpub}). Our next result is of a similar flavour. We show that any matrix of $T_m(K)^{(t-1)}$ can be written as sum of two elements of $p\left( T_m(K) \right) $, where $\ord=t$.

We know from \autoref{thm:order_t} and \autoref{eg:counterExample} that
$p\left(T_m(K)\right)\neq T_m(K)^{(t-1)}$, although $p\left(T_m(K)\right)\subseteq T_m(K)^{(t-1)}$. We want to show in this section that in all cases 
there exists $d=d(m,t)\in\mathbb{N}$ such that 
\[\underbrace{p\left(T_m(K)\right)+p\left(T_m(K)\right)+\ldots+p\left(T_m(K)\right)}_{d\text{-times}}=T_m(K)^{(t-1)},\]
under the conditions of \autoref{thm:order_t}. 
Since $\ord=t$ and $p\left(T_m(K)\right)\neq 0$, we get the existence of $1\leq i_1,i_2,\cdots,i_t\leq m$ such that
\[
p_{i_1i_2\ldots i_t}\neq 0.
\]
Consider the number $I_p=\min\left\{\vert \{i_1,i_2,\cdots,i_t\} \vert ; p_{i_1i_2\ldots i_t}\neq 0\right\}$. We will give several results, depending on the quantity $I_p$. 
We first analyze some examples for a few cases. These will be generalized in further discussions.
 
\begin{example}
    Take $m=6,~t=2$ and $I_p=1$, that is,  $p_{i_0i_0} \neq 0 $ for some $i_0$.
    \[
        X =  \begin{pmatrix}
                0 & 0 & x_{13} & x_{14} & x_{15} &  x_{16} \\
                0 & 0 & 0 & x_{24} & x_{25} &  x_{26} \\
                0 & 0 & 0 & 0 & x_{35} &  x_{36} \\
                0 & 0 & 0 & 0 & 0 &  x_{46} \\
                0 & 0 & 0 & 0 & 0 &  0 \\
                0 & 0 & 0 & 0 & 0 &  0 \\
            \end{pmatrix}
    \]
    Choose $u_{i_0}$ to be the matrix
    \begin{align*}
        \left(
            \scalemath{0.89}{
                \begin{array}{cccccc}
                    a_{11} & 1 & 0 & 0 & 0 & 0 
                    \\
                    0 & a_{22} & \left( \lambda _{123}u_{12} \right) ^{-1}x_{13} & \left( \lambda _{124}u_{12} \right) ^{-1}\left( x_{14} - \beta _{14} \right) & \left( \lambda _{125}u_{12} \right) ^{-1} \left( x_{15}-\beta _{15} \right)   &  \left( \lambda _{126}u_{12} \right) ^{-1} \left( x_{16}-\beta _{16} \right)  
                    \\
                    0 & 0 & a_{33} & \left( \lambda_{234} u_{23} \right) ^{-1}  x_{24} & \left( \lambda_{235}u_{23} \right) ^{-1}  \left( x_{25}-\beta _{25} \right)   &  \left( \lambda _{236}u_{23} \right) ^{-1} \left( x_{26}-\beta _{26} \right) 
                    \\
                    0 & 0 & 0 & a_{44} & \left( \lambda _{345} u_{34} \right)^{-1}   x_{35} &   \left( \lambda _{346}u_{34} \right)^{-1}  \left( x_{36}-\beta _{36} \right)
                    \\
                    0 & 0 & 0 & 0 & a_{55} & \left( \lambda _{456}u_{45} \right) ^{-1}  x_{46} 
                    \\
                    0 & 0 & 0 & 0 & 0 &  a_{66}
                \end{array}
            }
        \right)
    \end{align*}
    for a suitable choice of $\beta_{ij}$ and 
    \[
        \lambda _{ijk} = p_{i_0i_0}\left( \avec{ijk} \right) \neq 0.
    \]
\end{example} 

\hf In general, for general $m$ and $t=2$ we can choose the matrix $u_{i_0}$ as follows if $p_{i_0i_0}\neq 0$ for some $i_0$. For $2 \leq l \leq m-1$ and $1\leq s\leq m-l$, choose the entries to be 
\begin{align*}
    u_{12} & =  1 \\
    u_{l,l+1} & = \left( \lambda _{l-1,l,l+1}\cdot u_{l-1,l} \right) ^{-1}  \left[ x_{l-1,l+1} \right] 
    \\
    u_{l,l+2} & = \left( \lambda _{l-1,l,l+2} \cdot u_{l-1,l}\right) ^{-1} \left[ x_{l-1,l+2} - \beta _{l-1,l+2} \right] 
    \\
    \vdots & \kern 2cm \vdots \kern 2cm \vdots
    \\
    u_{l,l+s} & = \left( \lambda _{l-1,l,l+s} \cdot u_{l-1,l} \right) ^{-1} \left[ x_{l-1,l+s} - \beta _{l-1,l+s} \right].
\end{align*}

\noindent The other entries of $u_{i_0}$ are taken to be zero. For all $i\neq i_0$ we take $u_i = \diag \left( a_{11}(i), a_{22}(i), \ldots ,a_{mm}(i)\right) $.

\begin{example}
    Take $m=6,~t=3$ and $I_p=1$, that is,  $p_{i_0i_0i_0} \neq 0 $ for some $i_0$.
    \[
        X =  \begin{pmatrix}
                0 & 0 & 0 & x_{14} & x_{15} &  x_{16} \\
                0 & 0 & 0 & 0 & x_{25} &  x_{26} \\
                0 & 0 & 0 & 0 & 0 &  x_{36} \\
                0 & 0 & 0 & 0 & 0 &  0 \\
                0 & 0 & 0 & 0 & 0 &  0 \\
                0 & 0 & 0 & 0 & 0 &  0 \\
            \end{pmatrix}
    \]

    Choose the matrix $u_{i_0}$ as follows. 
    \begin{align*}
        \left(
            \scalemath{0.91}{
                \begin{array}{cccccc}
                    a_{11} & 1 & 0 & 0 & 0 & 0 
                    \\
                    0 & a_{22} & 1 & 0 & 0 & 0
                    \\
                    0 & 0 & a_{33} & \left( \lambda_{1234} \cdot u_{12}u_{23} \right) ^{-1} x_{14} & \left( \lambda _{1235}\cdot u_{12}u_{23} \right)^{-1}  \left( x_{15}-\beta _{15} \right)    &  \left( \lambda _{1236}\cdot u_{12}u_{23} \right) ^{-1}  \left( x_{16} - \beta _{16} \right) 
                    \\
                    0 & 0 & 0 & a_{44} & \left( \lambda _{2345}\cdot u_{23}u_{34}  \right)^{-1} x_{25}  &  \left( \lambda _{2346}\cdot u_{23}u_{34} \right) ^{-1} \left( x_{26}-\beta _{26} \right) 
                    \\
                    0 & 0 & 0 & 0 & a_{55} & \left( \lambda _{3456}\cdot u_{34}u_{45} \right) ^{-1}  x_{36} 
                    \\
                    0 & 0 & 0 & 0 & 0 &  a_{66} 
                    \\
                \end{array}
            } 
        \right)
    \end{align*}
    for a suitable choice of $\beta_{ij}$ and 
    \[
        \lambda _{ijkl} = p_{i_0i_0i_0}\left( \avec{ijkl} \right) \neq 0.
    \]
\end{example}

\hf In general, for general $m$ and $t=3$ we can choose the matrix $u_{i_0}$ as follows if $p_{i_0i_0i_0}\neq 0$ for some $i_0$. For $3 \leq l \leq m-1$ and $1\leq s\leq m-l$, choose the entries to be 
\begin{align*}
    u_{12} & = u_{23} = 1  \\
    u_{l,l+1} & = \left( \lambda_{l-2, l-1, l, l+1}\cdot u_{l-2,l-1}u_{l-1,l} \right) ^{-1} \left[ x_{l-2,l+1} \right] 
    \\
    u_{l,l+2} & = \left( \lambda _{l-2,l-1,l,l+2} \cdot u_{l-2,l-1}u_{l-1,l} \right) ^{-1} \left[ x_{l-2,l+2 } - \beta _{l-2,l+2} \right] 
    \\
    \vdots & \kern 2cm \vdots \kern 2cm \vdots
    \\
    u_{l,l+s} & = \left(\lambda _{l-2,l-1,l,l+s}\cdot u_{l-2,l-1}u_{l-1,l}\right) ^{-1} .\left[ x_{l-2,l+s} - \beta _{l-2,l+s} \right] .
\end{align*} 

\noindent The other entries of $u_{i_0}$ are taken to be zero. For all $i\neq i_0$ we take $u_i = \diag \left( a_{11}(i), a_{22}(i), \ldots ,a_{mm}(i)\right) $.

\begin{example}
    Take $m=6,~t=4$ and $I_p=1$, that is, $p_{i_0i_0i_0i_0i_0} \neq 0 $ for some $i_0$.
    \[
        X =  \begin{pmatrix}
                0 & 0 & 0 & 0 & x_{15} &  x_{16} \\
                0 & 0 & 0 & 0 & 0 &  x_{26} \\
                0 & 0 & 0 & 0 & 0 &  0 \\
                0 & 0 & 0 & 0 & 0 &  0 \\
                0 & 0 & 0 & 0 & 0 &  0 \\
                0 & 0 & 0 & 0 & 0 &  0 \\
            \end{pmatrix}
    \]

    Choose the matrix $u_{i_0}$ as follows.
    \begin{align*}
        \left(
        \begin{array}{cccccc}
            a_{11} & 1 & 0 & 0 & 0 & 0 
            \\
            0 & a_{22} & 1 & 0 & 0 & 0
            \\
            0 & 0 & a_{33} & 1 & 0 & 0
            \\
            0 & 0 & 0 & a_{44} & \left( \lambda _{12345}\cdot u_{12}u_{23}u_{34} \right)^{-1} x_{15}  &  \left( \lambda _{12346}\cdot u_{12}u_{23}u_{34} \right) ^{-1} \left( x_{16}-\beta _{16} \right) 
            \\
            0 & 0 & 0 & 0 & a_{55} & \left( \lambda _{23456}\cdot u_{23} u_{34}u_{45} \right) ^{-1}  x_{26} 
            \\
            0 & 0 & 0 & 0 & 0 &  a_{66} 
            \\
        \end{array} 
        \right)
    \end{align*}
    for a suitable choice of $\beta_{ij}$ and 
    \[
        \lambda _{i_1i_2i_3i_4i_5} = p_{i_0i_0i_0i_0i_0}\left( \avec{i_1i_2i_3i_4i_5} \right) \neq 0.
    \]
\end{example}

\hf In general, for general $m$ and $t=4$ we can choose the matrix $u_{i_0}$ as follows if $p_{i_0i_0i_0i_0}\neq 0$ for some $i_0$. For $4 \leq l \leq m-1$ and $1\leq s\leq m-l$, choose the entries to be 
\begin{align*}
    u_{12} & = u_{23} = u_{34} = 1 \\
    u_{l,l+1} & = \left( \lambda_{l-3,l-2, l-1, l, l+1}\cdot u_{l-3,l-2}u_{l-2,l-1}u_{l-1,l} \right) ^{-1} \left[ x_{l-3,l+1} \right] 
    \\
    u_{l,l+2} & = \left( \lambda _{l-3,l-2,l-1,l,l+2} \cdot u_{l-3,l-2}u_{l-2,l-1}u_{l-1,l} \right) ^{-1} \left[ x_{l-3,l+2 } - \beta _{l-3,l+2} \right] 
    \\
    \vdots & \kern 2cm \vdots \kern 2cm \vdots
    \\
    u_{l,l+s} & = \left(\lambda _{l-3,l-2,l-1,l,l+s}\cdot u_{l-3,l-2}u_{l-2,l-1}u_{l-1,l}\right) ^{-1} .\left[ x_{l-3,l+s} - \beta _{l-3,l+s} \right] .
\end{align*} 

\noindent The other entries of $u_{i_0}$ are taken to be zero. For all $i\neq i_0$ we take $u_i = \diag \left( a_{11}(i), a_{22}(i), \ldots ,a_{mm}(i)\right) $.

\vspace{0.3cm}
\hf For general $m$ and $t\geq 2$ we can choose the matrix $u_{i_0}$ as follows if $p_{\underbrace{i_0i_0\ldots i_0}_{t-\text{times}}}\neq 0$ for some $i_0$. For $ t \leq l \leq m-1$ and $1\leq s\leq m-l$, choose the entries to be 
\begin{align*}
    u_{12} & = u_{23} = \cdots = u_{t-1,t} = 1\\
    u_{l,l+1} & = \left( \lambda_{l-t+1, \ldots, l, l+1}\cdot u_{l-t+1,l-t+2}\cdots u_{l-1,l} \right) ^{-1} \left[ x_{l-t+1,l+1} \right] 
    \\
    u_{l,l+2} & = \left( \lambda _{l-t+1,\ldots, l,l+2} \cdot u_{l-t+1,l-t+2}\cdots u_{l-1,l} \right) ^{-1} \left[ x_{l-t+1,l+2 } - \beta _{l-t+1,l+2} \right] 
    \\
    \vdots & \kern 2cm \vdots \kern 2cm \vdots \kern 4cm \vdots
    \\
    u_{l,l+s} & = \left(\lambda _{l-t+1,\ldots ,l,l+s}\cdot u_{l-t+1,l-t+2}\cdots u_{l-1,l}\right) ^{-1} .\left[ x_{l-t+1,l+s} - \beta _{l-t+1,l+s} \right] .
\end{align*}

\noindent The other entries of $u_{i_0}$ are taken to be zero. For all $i\neq i_0$ we take $u_i = \diag \left( a_{11}(i), a_{22}(i), \ldots ,a_{mm}(i)\right) $.

\hf Therefore, for a polynomial $p$ of order $t$, if the entries of $X \in T_m(K)^{(t-1)}$ satisfy $x_{i,i+l} \neq 0$ for $1\leq i\leq m-t$ and $t\leq l \leq m-i$, then $X\in p\left( T_m(K) \right) $. Hence, we have proved the following proposition.

\begin{proposition}\label{prop:Waring_all_i_same}
    Suppose $\ord = t\geq 2$ and $I_p=1$. Then $p \left(\tk[m]\right) +  p \left(\tk[m]\right) = T_m(K)^{(t-1)}$.  
\end{proposition}
Next we turn to the case when $|\{i_1,i_2,\ldots,i_t\}|=t$ and $p_{i_{\sigma(1)}i_{\sigma(2)}\cdots i_{\sigma(1)}}\neq 0$ for all $\sigma\in\langle(1,2,\ldots,t)\rangle$. We start with two examples for the cases $(m,t)=(6,2)$ and $(m,t)=(6,3)$.
\begin{example} Let us consider the first case, $m=6,~t=2$, $I_p = 2$, and $p_{i_0j_0}\cdot p_{j_0i_0} \neq 0$. 
Then using \autoref{prop:ifOneNonzeroThenManynonzero}, we can choose suitable elements of $K^6$ such that
\begin{align*}
    \lambda_{\sigma(1)\sigma(2)\sigma(3)}=p_{i_0j_0}\left(\avec{\sigma(1)\sigma(2)\sigma(3)}\right)\neq 0\\
    \text{and }\mu_{\sigma(1)\sigma(2)\sigma(3)}=p_{j_0i_0}\left(\avec{\sigma(1)\sigma(2)\sigma(3)}\right)\neq 0,
\end{align*}
for all $\sigma\in S_6$. Now let us choose the following $X\in T_6(K)^{(1)}$ with $x_{ij}\neq 0$. We claim that $X\in p\left(T_6(K)\right)$. We have,
    \[
   X =  \begin{pmatrix}
        0 & 0 & x_{13} & x_{14} & x_{15} &  x_{16} \\
        0 & 0 & 0 & x_{24} & x_{25} &  x_{26} \\
        0 & 0 & 0 & 0 & x_{35} &  x_{36} \\
        0 & 0 & 0 & 0 & 0 &  x_{46} \\
        0 & 0 & 0 & 0 & 0 &  0 \\
        0 & 0 & 0 & 0 & 0 &  0 \\
    \end{pmatrix}
\]
For $i=i_0$ we choose,
\color{black}
\begin{align*}
    u_{i}=\left(
    \begin{array}{cccccc}
        a_{11} & 1 & 0 & 0 & 0 & 0 
        \\
        0 & a_{22} & 0 & 0 & 0 &  0
        \\
        0 & 0 & a_{33} & \left( \mu_{234}u_{23j} \right)^{-1} x_{24} & \left( \mu _{235}u_{23j} \right) ^{-1} \left[ x_{25}- \beta _{25} \right]  & \left( \mu _{236}u_{23j} \right) ^{-1} \left[ x_{26}- \beta _{26} \right] 
        \\
        0 & 0 & 0 & a_{44} & 0 & 0
        \\
        0 & 0 & 0 & 0 & a_{55} & \left( \mu _{456} u_{45j} \right) ^{-1} x_{46}
        \\
        0 & 0 & 0 & 0 & 0 &  a_{66} 
        \\
    \end{array} 
    \right),
\end{align*}  
\color{black}
and for $j=j_0$ we take $u_j$ to be the matrix
\color{black}
\begin{align*}
    \left(
        \scalemath{0.86}{
            \begin{array}{cccccc}
                a_{11} & 0 & 0 & 0 & 0 & 0 
                \\
                0 & a_{22} & \left( \lambda _{123}u_{12i} \right)^{-1}x_{13}  & \left( \lambda _{124}u_{12i} \right) ^{-1}  \left[ x_{14} - \beta _{14} \right]  & \left( \lambda _{125}u_{12i} \right) ^{-1} \left[ x_{15} - \beta _{15} \right]  &  \left( \lambda _{126}u_{12i} \right) ^{-1}  \left[ x_{16}- \beta _{16} \right] 
                \\
                0 & 0 & a_{33} &  0 & 0 & 0
                \\
                0 & 0 & 0 & a_{44} & \left( \lambda _{345}u_{34i} \right) ^{-1} x_{35} & \left( \lambda _{346}u_{34i} \right) ^{-1}  \left[ x_{36}-\beta _{36} \right] 
                \\
                0 & 0 & 0 & 0 & a_{55} & 0
                \\
                0 & 0 & 0 & 0 & 0 &  a_{66} 
            \end{array}
        } 
    \right),
\end{align*}
\color{black}
for suitable choice of $\beta_{ab}$'s.
For all other $i\neq i_0,j_0$ we take $u_i=\diag\left(a_{11}(i),a_{22}(i),\ldots, a_{66}(i)\right)$. It can be checked that with the given choice we get that $X\in p\left(T_6(K)\right)$.
\end{example}
\begin{example}
    
Let us consider the next case, $m=6,~t=3$, $I_p=3$,  and $p_{i_0j_0k_0}\cdot p_{j_0k_0i_0}\cdot p_{k_0i_0j_0}\neq 0$. Then using \autoref{prop:ifOneNonzeroThenManynonzero}, we can choose suitable elements of $K^6$ such that
\begin{align*}
    \lambda^{abc}_{\sigma(1)\sigma(2)\sigma(3)\sigma(4)}=p_{abc}\left(\avec{\sigma(1)\sigma(2)\sigma(3)\sigma(4)}\right)\neq 0,
\end{align*}
for all $\sigma\in S_6$ and $(a,b,c)\in\{(i_0,j_0,k_0),(j_0,k_0,i_0),(k_0,i_0,j_0)\}$. Now let us choose the following $X\in T_6(K)^{(2)}$ with $x_{ij}\neq 0$. We claim that $X\in p\left(T_6(K)\right)$. We have the matrix,
\color{black}

\[
   X =  \begin{pmatrix}
        0 & 0 & 0 & x_{14} & x_{15} &  x_{16} \\
        0 & 0 & 0 & 0 & x_{25} &  x_{26} \\
        0 & 0 & 0 & 0 & 0 &  x_{36} \\
        0 & 0 & 0 & 0 & 0 &  0 \\
        0 & 0 & 0 & 0 & 0 &  0 \\
        0 & 0 & 0 & 0 & 0 &  0 \\
    \end{pmatrix}. 
\]
For making the entries a little neater, we use the notation $u_{abw}$ to denote the $ab$-th entry of the matrix $u_w$. For $i=i_0$ take $u_{i_0}$ to be,
\color{black}
\begin{align*}
    \left(
        \begin{array}{cccccc}
        a_{11} & 1 & 0 & 0 & 0 & 0
        \\
        0 & a_{22} & 0 & 0 & 0 & 0
        \\
        0 & 0 & a_{33} & 0 & 0 & 0
        \\
        0 & 0 & 0 & a_{44} & \left( \lambda ^{jki}_{2345}u_{23j}u_{34k} \right) ^{-1} x_{25} & \left( \lambda ^{jki}_{2346}u_{23j}u_{34k} \right) ^{-1} \left[ x_{26}-\beta _{26} \right]  
        \\
        0 & 0 & 0 & 0 & a_{55} & 0
        \\
        0 & 0 & 0 & 0 & 0 &  a_{66}
    \end{array} 
    \right),
\end{align*}
\color{black}
for $j=j_0$ the matrix $u_{j_0}$ to be
\color{black}
\begin{align*}
    \left(
    \begin{array}{cccccc}
        a_{11} & 0 & 0 & 0 & 0 & 0 
        \\
        0 & a_{22} & 1 & 0 & 0 & 0 
        \\
        0 & 0 & a_{33} & 0 & 0 & 0 
        \\
        0 & 0 & 0 & a_{44} & 0 & 0 
        \\
        0 & 0 & 0 & 0 & a_{55} &  \left( \lambda ^{kij}_{3456}u_{34k}u_{45i} \right) ^{-1}  x_{36}
        \\
        0 & 0 & 0 & 0 & 0 &  a_{66} 
    \end{array} 
    \right)
\end{align*}   
\color{black}
and finally for $k=k_0$ the matrix $u_{k_0}$ to be
\color{black}
\begin{align*}
    \left(
        \scalemath{0.9}{
            \begin{array}{cccccc}
                a_{11} & 0 & 0 & 0 & 0 & 0 
                \\
                0 & a_{22} & 0 & 0 & 0 & 0 
                \\
                0 & 0 & a_{33} & \left( \lambda _{1234}^{ijk} u_{12i}u_{23j} \right) ^{-1} x_{14} & \left( \lambda ^{ijk}_{1235}u_{12i}u_{23j} \right) ^{-1} \left[ x_{15}-\beta _{15} \right]  & \left( \lambda ^{ijk}_{1236}u_{12i}u_{23j} \right) ^{-1} \left[ x_{16}-\beta _{16} \right]
                \\
                0 & 0 & 0 & a_{44} & 0 & 0
                \\
                0 & 0 & 0 & 0 & a_{55} & 0 
                \\
                0 & 0 & 0 & 0 & 0 &  a_{66} 
                \\
            \end{array}
        } 
    \right),
\end{align*}
\color{black}
for suitable choice of $\beta_{ab}$'s.
For all other $i\neq i_0,j_0,k_0$ we take $u_i=\diag\left(a_{11}(i),a_{22}(i),\ldots, a_{66}(i)\right)$. 
It can be checked that with the given choice, we get that $X\in p\left(T_6(K)\right)$.
\end{example}
Note that in the above two examples we get that $p\left(T_m(K)\right)+p\left(T_m(K)\right)=T_{m}(K)^{t-1}$. So let us now work out the general case of $m$ and $t$.  Hereafter, the notation $u_w(a,b)$ will be used to denote the $(a,b)^{\text{th}}$ entry of the matrix $u_w$.

\color{black}
 Remember that for all $i\neq i_j$, $1\leq j\leq t$ we are fixing $u_i=\diag(a_{11}(i),a_{22}(i),\ldots, a_{mm}(i))$. Denote 
 \[
    p_{a_1,a_2,\ldots ,a_t} \left( \avec{\alpha _1,\alpha _2,\ldots \alpha _{t+1}} \right) = \lambda^{a_1,a_2,\ldots, a_t}_{\alpha _1,\alpha _2,\ldots \alpha _{t+1}}.
 \]

Note that using \autoref{prop:ifOneNonzeroThenManynonzero}, we fix all the concentered $\lambda$'s to be nonzero. Now we start with the recursive steps. 

\vspace{0.2cm}
\noindent Step $1$ : We fix all $(l,l+1)$-th entries of $u_i$'s.in this step. This step is further divided into several steps.

\noindent Step $(1,1)$ : We fix 
\[
u_{i_j}(j,j+1)=\begin{cases}
1 & \text{ if }1\leq j\leq t-1\\
\left(\lambda^{i_1,i_2,\ldots,i_{t-1},i_{t}}_{1,2,\ldots,t,t+1}\right)^{-1}
x_{1,t+1} & \text{ if }j=t
\end{cases}.
\]

\noindent Step $(1,2)$ : We take $u_{i_1}(t+1,t+2)=\left(\lambda^{i_2,i_3,\ldots,i_{t},i_{1}}_{2,3\ldots,t+1,t+2}u_{i_t}(t,t+1)\right)^{-1}
x_{2,t+2}$.

\noindent Step $(1,3)$ : We take $u_{i_2}(t+2,t+3)=\left(\lambda^{i_3,i_4,\ldots,i_{1},i_{2}}_{3,4\ldots,t+2,t+3}u_{i_t}(t,t+1)u_{i_1}(t+1,t+2)\right)^{-1}
x_{3,t+3}$.

\vspace{0.2cm}
\noindent In general we take 
\begin{align*}
    u_{i_l} \left(kt+l,kt+l+1\right) & = \left(\lambda^{i_{l+1},i_{l+2},\ldots,i_{l}}_{(k-1)t+l+1,(k-1)t+l+2,\ldots,kt+l+1} \right. 
    \\
    & \kern 1cm 
    \cdot u_{i_{l+1}}\left((k-1)t+l+1,(k-1)t+l+2\right) \\
    & \kern 1cm \cdot u_{i_{l+2}}\left((k-1)t+l+2,(k-1)t+l+3\right) 
    \\
    & \kern 1cm \left. \cdot \cdots \cdot u_{i_{l-1}}\left(kt+l-1,kt+l\right) \right)^{-1}  \cdot x_{(k-1)t+l+1, kt+l+1}.
\end{align*}
Here if $l+s \geq t$, then we take $\text{ mod }t$ value with the assumption that $0$ will be replaced by $t$. For example, if $m=8$ and $t=3$ then $\lambda^{i_5,i_6,i_7}_{5,6,7,8}$ will be replaced by $\lambda ^{i_2,i_3,i_1}_{5,6,7,8}$. 

\noindent Step $(1,m-2t+2)$ : All other $(l,l+1)$ entries of $u_{i_j}$ for $1\leq j\leq t$ will be taken to be zero. 

\noindent Step $2$ : We fix all $(l,l+2)$-th entries of $u_i$'s.in this step. This step is further divided into several steps.

\noindent Step $(2,1)$: Fix $u_{i_t}(t,t+2)=
\left(\lambda^{i_1,i_2,\ldots,i_{t-1},i_{t}}_{1,2,\ldots,t,t+2}\right)^{-1}
\left(x_{1,t+2}-\beta_{1,t+2}\right)$
for a suitable $\beta_{1,t+2}\in K$.

\noindent Step $(2,2)$: We take $u_{i_1}(t+1,t+3)=\left(\lambda^{i_2,i_3,\ldots,i_{t},i_{1}}_{2,3\ldots,t+1,t+3}u_{i_t}(t,t+1)\right)^{-1}
\left(x_{2,t+3}-\beta_{2,t+3}\right)$
for a suitable $\beta_{2,t+3}\in K$.

\noindent Step $(2,3)$: We take $u_{i_2}(t+2,t+4)=\left(\lambda^{i_2,i_3,\ldots,i_{t},i_{1}}_{2,3\ldots,t+1,t+3}u_{i_t}(t,t+1)u_{i_1}(t+1,t+2)\right)^{-1}
\left(x_{3,t+4}-\beta_{3,t+4}\right)$
for a suitable $\beta_{3,t+4}\in K$.

\vspace{0.2cm}
\noindent In general we take 
\begin{align*}
    u_{i_l} \left(kt+l,kt+l+2\right) & = \left(\lambda^{i_{l+1},i_{l+2},\ldots,i_{l}}_{(k-1)t+l+1,(k-1)t+l+2,\ldots,kt+l,kt+l+2} \right. 
    \\
    & \kern 1cm 
    \cdot u_{i_{l+1}}\left((k-1)t+l+1,(k-1)t+l+2\right) \\
    & \kern 1cm \cdot u_{i_{l+2}}\left((k-1)t+l+2,(k-1)t+l+3\right) 
    \\
    & \kern 1cm \left. \cdot \cdots \cdot u_{i_{l-1}}\left(kt+l-1,kt+l\right) \right)^{-1} \\
    & \kern 1cm \left(x_{(k-1)t+l+1, kt+l+2}-  \beta _{(k-1)t+l+1, kt+l+2} \right)
\end{align*}
for a suitable $\beta _{(k-1)t+l+1, kt+l+2}\in K$. For $l+s$, the convention will be used as previously. 

\vspace{0.2cm}
\noindent In general, in the $T^{\text{th}}$-step we choose
\begin{align*}
    u_{i_l} \left(kt+l,kt+l+T\right) & = \left(\lambda^{i_{l+1},i_{l+2},\ldots,i_{l}}_{(k-1)t+l+1,(k-1)t+l+2,\ldots,kt+l,kt+l+T} \right. 
    \\
    & \kern 1cm 
    \cdot u_{i_{l+1}}\left((k-1)t+l+1,(k-1)t+l+2\right) \\
    & \kern 1cm \cdot u_{i_{l+2}}\left((k-1)t+l+2,(k-1)t+l+3\right) 
    \\
    & \kern 1cm \left. \cdot \cdots \cdot u_{i_{l-1}}\left(kt+l-1,kt+l\right) \right)^{-1} \\
    & \kern 1cm \left(x_{(k-1)t+l+1, kt+l+T}-  \beta _{(k-1)t+l+1, kt+l+T} \right)
\end{align*}
for a suitable $\beta _{(k-1)t+l+1, kt+l+T}\in K$. For $l+s$, the convention will be used as previously. 

\color{black}
Hence we have the following proposition.
\begin{proposition}\label{prop:Waring_all_sigma_nonzero}
    Suppose $\ord = t$, $I_p=t$ and $p_{i_{\sigma (1)},i_{\sigma (2)},\ldots, i_{\sigma (t)}}$ is not identically zero for $\sigma \in \left\langle (1,\ldots ,t) \right\rangle $.  Then $p \left(\tk[m]\right) +  p \left(\tk[m]\right) = T_m(K)^{(t-1)}$.  
\end{proposition}

\hf We continue with some more examples before presenting the final proposition of this section. 

\begin{example}
    Consider the case $m=7$ and $t=3$. Let
    \[
        X = 
        \begin{pmatrix}
            0 & 0 & 0 & x_{14} & x_{15} & x_{16} &  x_{17} \\
            0 & 0 & 0 & 0 & x_{25} & x_{26} &  x_{27} \\
            0 & 0 & 0 & 0 & 0 & x_{36} &  x_{37} \\
            0 & 0 & 0 & 0 & 0 & 0 &  x_{47} \\
            0 & 0 & 0 & 0 & 0 & 0 & 0  \\
            0 & 0 & 0 & 0 & 0 & 0 & 0  \\
            0 & 0 & 0 & 0 & 0 & 0 & 0  \\
        \end{pmatrix}.
    \]
    We will show that the following matrices are in the image $p\left( T_7(K) \right) $. 
    \[
        X_1 = 
        \begin{pmatrix}
             R_1(X) \\
             0 \\
             0 \\
             R_4(X) \\
             0 \\
             0 \\
             0 \\
        \end{pmatrix}, 
        X_2 = 
        \begin{pmatrix}
             0 \\
             R_2(X) \\
             0 \\
             0 \\
             0 \\
             0 \\
             0 \\
        \end{pmatrix},
        \text{ and }
        X_3 = 
        \begin{pmatrix}
             0 \\
             0 \\
             R_3(X) \\
             0 \\
             0 \\
             0 \\
             0 \\
        \end{pmatrix},  
    \]
    where $R_i(X)$ is the $i^{\text{th}}$ row of the matrix $X$. At first let us show that $X_1 \in P\left( T_7(K) \right) $. Since $p\left( T_4(K) \right)\neq 0 $, there exists $i_1,i_2,i_3$ such that $p_{i_1,i_2,i_3} \neq 0$. Finally, using \autoref{prop:ifOneNonzeroThenManynonzero} we choose existence of $\vec{\mathbf{a}}_{11},\vec{\mathbf{a}}_{22},\ldots ,\vec{\mathbf{a}}_{77}$ such that
    \[
        \lambda ^{i_1,i_2,i_3}_{\sigma (1).\sigma (2),\sigma (3),\sigma (4)} = p_{i_1,i_2,i_3}\left( \avec{\sigma (1),\sigma (2),\sigma (3),\sigma (4)} \right) \neq 0
    \]
    for all $\sigma \in S_7$. 
    Consider
    \[
        u_{i_1}^{(1)} = 
        \begin{pmatrix}
            a_{11}(i_1) & 1 & 0 & 0 & 0 & 0 &  0 \\
            0 & a_{22}(i_1) & 0 & 0 & 0 & 0 &  0 \\
            0 & 0 & a_{33}(i_1) & 0 & 0 & 0 &  0 \\
            0 & 0 & 0 & a_{44}(i_1) & 1 & 0 &  0 \\
            0 & 0 & 0 & 0 & a_{55}(i_1) & 0 &  0 \\
            0 & 0 & 0 & 0 & 0 & a_{66}(i_1) &  0 \\
            0 & 0 & 0 & 0 & 0 & 0 &  a_{77}(i_1) \\
        \end{pmatrix},
    \] 

    \[
        u_{i_2}^{(1)} = 
        \begin{pmatrix}
            a_{11}(i_1) & 0 & 0 & 0 & 0 & 0 &  0 \\
            0 & a_{22}(i_1) & 1 & 0 & 0 & 0 &  0 \\
            0 & 0 & a_{33}(i_1) & 0 & 0 & 0 &  0 \\
            0 & 0 & 0 & a_{44}(i_1) & 0 & 0 &  0 \\
            0 & 0 & 0 & 0 & a_{55}(i_1) & 1 &  0 \\
            0 & 0 & 0 & 0 & 0 & a_{66}(i_1) &  0 \\
            0 & 0 & 0 & 0 & 0 & 0 &  a_{77}(i_1) \\
        \end{pmatrix}\text{ and }
    \] 

    \[
        u_{i_3}^{(1)} = 
        \scalemath{0.72}{
            \begin{pmatrix}
                a_{11}(i_1) & 0 & 0 & 0 & 0 & 0 &  0 \\
                0 & a_{22}(i_1) & 0 & 0 & 0 & 0 &  0 \\
                0 & 0 & a_{33}(i_1) & \left( \lambda ^{i_1,i_2,i_3}_{1,2,3,4} \right) ^{-1} x_{14} & \left( \lambda ^{i_1,i_2,i_3}_{1,2,3,5} \right) ^{-1} \left( x_{15}-\beta _{15} \right)  & \left( \lambda ^{i_1,i_2,i_3}_{1,2,3,6} \right) ^{-1} \left( x_{16}-\beta _{16} \right)  &  \left( \lambda ^{i_1,i_2,i_3}_{1,2,3,7} \right) ^{-1} \left( x_{17}-\beta _{17} \right) \\
                0 & 0 & 0 & a_{44}(i_1) & 0 & 0 &  0 \\
                0 & 0 & 0 & 0 & a_{55}(i_1) & 0 &  0 \\
                0 & 0 & 0 & 0 & 0 & a_{66}(i_1) & \left( \lambda ^{i_1,i_2,i_3}_{4,5,6,7} \right) ^{-1} x_{47} \\
                0 & 0 & 0 & 0 & 0 & 0 &  a_{77}(i_1) \\
            \end{pmatrix}
        },
    \] 
    for suitable choice of $\beta$'s. All other $u_{i}^{(1)}$ are taken to  be $\diag\left( a_{11}(i),a_{22}(i),\ldots ,a_{77}(i) \right) $.  If any of the two $i_j$'s are equal (say $i_0 $), then we choose $u_{i_0}^{(1)}$ to be the sum of $u_{i_j}^{(1)}$'s. For example, if $i_1=i_2=i_0$, then we take $u_{i_0}^{(1)} = u_{i_1}^{(1)} + u_{i_2}^{(1)}$, and $u_{i_3}$ remains same. Note that, with the above choice of matrices $p\left( \vec{u} \right) = X_1 $. Similarly, 

    \[
        u_{i_1}^{(2)} = 
        \begin{pmatrix}
            a_{11}(i_1) & 0 & 0 & 0 & 0 & 0 &  0 \\
            0 & a_{22}(i_1) & 1 & 0 & 0 & 0 &  0 \\
            0 & 0 & a_{33}(i_1) & 0 & 0 & 0 &  0 \\
            0 & 0 & 0 & a_{44}(i_1) & 0 & 0 &  0 \\
            0 & 0 & 0 & 0 & a_{55}(i_1) & 0 &  0 \\
            0 & 0 & 0 & 0 & 0 & a_{66}(i_1) &  0 \\
            0 & 0 & 0 & 0 & 0 & 0 &  a_{77}(i_1) \\
        \end{pmatrix},
    \] 

    \[
        u_{i_2}^{(2)} = 
        \begin{pmatrix}
            a_{11}(i_1) & 0 & 0 & 0 & 0 & 0 &  0 \\
            0 & a_{22}(i_1) & 0 & 0 & 0 & 0 &  0 \\
            0 & 0 & a_{33}(i_1) & 1 & 0 & 0 &  0 \\
            0 & 0 & 0 & a_{44}(i_1) & 0 & 0 &  0 \\
            0 & 0 & 0 & 0 & a_{55}(i_1) & 0 &  0 \\
            0 & 0 & 0 & 0 & 0 & a_{66}(i_1) &  0 \\
            0 & 0 & 0 & 0 & 0 & 0 &  a_{77}(i_1) \\
        \end{pmatrix}\text{ and }
    \] 

    \[
        u_{i_3}^{(2)} = 
        \scalemath{0.8}{
            \begin{pmatrix}
                a_{11}(i_1) & 0 & 0 & 0 & 0 & 0 &  0 \\
                0 & a_{22}(i_1) & 0 & 0 & 0 & 0 &  0 \\
                0 & 0 & a_{33}(i_1) & 0 & 0 & 0 & 0 \\
                0 & 0 & 0 & a_{44}(i_1) & \left( \lambda ^{i_1,i_2,i_3}_{2,3,4,5} \right) ^{-1} x_{25} & \left( \lambda ^{i_1,i_2,i_3}_{2,3,4,6} \right) ^{-1} \left( x_{26}-\beta _{26} \right)  & \left( \lambda ^{i_1,i_2,i_3}_{2,3,4,7} \right) ^{-1} \left( x_{27}-\beta _{27} \right) \\
                0 & 0 & 0 & 0 & a_{55}(i_1) & 0 &  0 \\
                0 & 0 & 0 & 0 & 0 & a_{66}(i_1) & 0 \\
                0 & 0 & 0 & 0 & 0 & 0 &  a_{77}(i_1) \\
            \end{pmatrix}
        },
    \] 
    for suitable choice of $\beta$'s. All other $u_{i}^{(2)}$ are taken to  be $\diag\left( a_{11}(i),a_{22}(i),\ldots ,a_{77}(i) \right) $. Similarly, as in the previous case (that is, considering the cases where two of them are equal), we get $X_2\in p\left( T_7(K) \right) $. Finally, taking 

    \[
        u_{i_1}^{(3)} = 
        \begin{pmatrix}
            a_{11}(i_1) & 0 & 0 & 0 & 0 & 0 &  0 \\
            0 & a_{22}(i_1) & 0 & 0 & 0 & 0 &  0 \\
            0 & 0 & a_{33}(i_1) & 1 & 0 & 0 &  0 \\
            0 & 0 & 0 & a_{44}(i_1) & 0 & 0 &  0 \\
            0 & 0 & 0 & 0 & a_{55}(i_1) & 0 &  0 \\
            0 & 0 & 0 & 0 & 0 & a_{66}(i_1) &  0 \\
            0 & 0 & 0 & 0 & 0 & 0 &  a_{77}(i_1) \\
        \end{pmatrix},
    \] 

    \[
        u_{i_2}^{(3)} = 
        \begin{pmatrix}
            a_{11}(i_1) & 0 & 0 & 0 & 0 & 0 &  0 \\
            0 & a_{22}(i_1) & 0 & 0 & 0 & 0 &  0 \\
            0 & 0 & a_{33}(i_1) & 0 & 0 & 0 &  0 \\
            0 & 0 & 0 & a_{44}(i_1) & 1 & 0 &  0 \\
            0 & 0 & 0 & 0 & a_{55}(i_1) & 0 &  0 \\
            0 & 0 & 0 & 0 & 0 & a_{66}(i_1) &  0 \\
            0 & 0 & 0 & 0 & 0 & 0 &  a_{77}(i_1) \\
        \end{pmatrix}\text{ and }
    \] 

    \[
        u_{i_3}^{(3)} = 
        \scalemath{0.8}{
            \begin{pmatrix}
                a_{11}(i_1) & 0 & 0 & 0 & 0 & 0 &  0 \\
                0 & a_{22}(i_1) & 0 & 0 & 0 & 0 &  0 \\
                0 & 0 & a_{33}(i_1) & 0 & 0 & 0 & 0 \\
                0 & 0 & 0 & a_{44}(i_1) & 0 & 0 & 0 \\
                0 & 0 & 0 & 0 & a_{55}(i_1) & \left( \lambda ^{i_1,i_2,i_3}_{3,4,5,6} \right) ^{-1} x_{36} & \left( \lambda ^{i_1,i_2,i_3}_{3,4,5,7} \right) ^{-1} \left( x_{37} - \beta_{37} \right) \\
                0 & 0 & 0 & 0 & 0 & a_{66}(i_1) & 0 \\
                0 & 0 & 0 & 0 & 0 & 0 &  a_{77}(i_1) \\
            \end{pmatrix}
        },
    \] 
    for suitable choice of $\beta$'s. All other $u_{i}^{(3)}$ are taken to  be $\diag\left( a_{11}(i),a_{22}(i),\ldots ,a_{77}(i) \right) $. Similarly, as in the previous cases (that is, considering the cases where two of them are equal), we get $X_3\in p\left( T_7(K) \right) $. 

    \vspace{0.2cm}
    \hf This proves that $X\in p\left( T_7(K) \right) + p\left( T_7(K) \right) + p\left( T_7(K) \right)$. Hence we have that
    \[
        T_7(K)^{(2)} = \underbrace{p\left( T_7(K) \right) + p\left( T_7(K) \right) + \cdots + p\left( T_7(K) \right) }_{6\text{-times}}.
    \] 
\end{example}
The above example motivates the following proposition. 
\begin{proposition}\label{prop:Waring-ord=t}
Let $p\in K[F_n]$ be a polynomial in non-commutative variables. Let $\ord =t$ and $I_p\geq 2$. 
Then 
\[
\underbrace{p\left(T_m(K)\right)+p\left(T_m(K)\right)+\ldots+p\left(T_m(K)\right)}_{d\text{-times}}=T_m(K)^{(t-1)},
\] where $d=\min\{t,m-t\}$.
\end{proposition}
\begin{proof}
Let us continue with the notations of the discussion at the beginning of this section.
    Let us write $X=X_1+X_2+\cdots+X_t$, with 
    \[R_b(X_a)=\begin{cases}
        R_b(X)&\text{ if }b\equiv a\pmod{t}\\
        0&\text{ otherwise}
    \end{cases},\]
\noindent where $R_b(A)$ denotes the $b^\text{th}$ row of the matrix $A$, for $1\leq a\leq t$. We claim that $X_a\in p\left(T_m(K)\right)$ for all $1\leq a\leq t$.
Note that if $m-t<t$, then in the above decomposition of $X$ we will be left with only $m-t$ many terms. This justifies the choice of $d$ in the theorem. It remains to prove our claim.
We go through the following steps for constructing the matrices $u_{i_j}^{(a)}$ for $1\leq a\leq d$. Firstly set the diagonal of $u_{i_j}$ to be $\diag(\Vec{a}_{11}(i_j),\Vec{a}_{22}(i_j),\ldots,\Vec{a}_{mm}(i_j))$. The next choices will be made in several steps.

\noindent Step $1$: Choose $u_{i_j}^{(a)}\left((k-1)t+j+a-1,(k-1)t+j+a\right)=1$, for all $1\leq j\leq t-1$, $1\leq a\leq \min{\{t,m-t\}}$ and a suitable choice of $k\geq 1$ and $kt+a\leq m$.
We set 
\[
u_{i_t}^{(a)}\left(kt+a-1,kt+a\right)=\left(\lambda^{i_1,i_2,\ldots,i_t}_{(k-1)t+a,(k-1)t+a+1,\ldots, kt+a}\right)^{-1}x_{(k-1)t+a,kt+a},
\]
for the choices of $k$ as above.
All the other $(c,c+1)$ entries of $u_{i_j}^{(a)}$ are taken to be zero. 

\noindent Step $T+1$ $(T\geq 1,~kt+a\leq m)$: For $kt+a+T\leq m$, 
We define
\begin{align*}
    u_{i_t}^{(a)}\left(kt+a-1,kt+a+T\right)=\left(\lambda^{i_1,i_2,\ldots,i_t}_{(k-1)t+a,(k-1)t+a+1,\ldots,kt+a-1,kt+a+T}\right)^{-1}x_{(k-1)t+a,kt+a+T}
\end{align*}
\color{black}
for a suitable choice of $\beta's$. All other entries of $U_{i_j}^{(a)}$ are taken to be zero. Now let us define the set $\Pi(j)=\{\gamma|i_{\gamma}=i_j\}$ and $\pi(j)=\min \Pi(j)$.
Further, take  $\Delta(p)=\{\pi(j)|1\leq j\leq t\}$. For $j\in \Delta(p)$, we define
\begin{align*}
    u_{i_j}^{(a)}=\sum\limits_{\gamma\in\Pi(j)}u_{i_{\gamma}}^{(a)}.
\end{align*}
For all $j\in\{1,2,\ldots,m\}\setminus \Delta(p)$, define $u_{i_j}^{(a)}=\diag(\Vec{a}_{11}(i_j),\Vec{a}_{22}(i_j),\ldots,\Vec{a}_{mm}(i_j))$.
Then with the prescribed choice we have that $X_a=p\left(u_{1}^{(a)},u_{2}^{(a)},\ldots,u_{n}^{(a)}\right)\in p\left(T_m(K)\right)$.
This proves that if all the $x_{ij}$'s of $X$ are nonzero then we have 
\[
X\in \underbrace{p\left(T_m(K)\right)+p\left(T_m(K)\right)+\ldots+p\left(T_m(K)\right)}_{\min \{t,m-t\}\text{-times}}.
\]
The result follows immediately.
\end{proof}
We note down here the following result in case $\ord=2$, where it can be shown $d=2$ without using \autoref{prop:Waring-ord=t}.
\begin{proposition}
    Let $p\in K[F_n]$ be a polynomial in non-commutative variables. Suppose $\ord = 2$. Then $p \left(\tk[m]\right) +  p \left(\tk[m]\right) = T_m(K)^{(1)}$. 
\end{proposition}
\begin{proof}
Assume $X= \left(x_{ab}\right) \in T_m(K)^{(1)}$ satisfies $x_{ab}\neq 0$ for all $b-a\ge 2$. We want to show that $X\in p\left(T_M(K)\right)$. Continuing with the notation of this section, we see that $I_p=1$ or $2$. If $I_p=1$, then the result follows from \autoref{prop:Waring_all_i_same}. 
If $I_p=2$ and $p_{i_1i_2}p_{i_2i_1}\neq 0$, then the result follows from \autoref{prop:Waring_all_sigma_nonzero}. 
Hence we are left with the case when either $p_{i_0j_0}$ or $p_{j_0i_0}$ is the zero function. Without loss of generality, we assume that $p_{i_0j_0}$ to be nonzero. We can assume that $p_{i_0i_0}$ and $p_{j_0j_0}$ are the zero polynomials. Using \autoref{prop:ifOneNonzeroThenManynonzero}, we assume that 
\[
    \lambda_{\sigma(1),\sigma(2)} = p_{i_0j_0}\left(\avec{\sigma(1),\sigma(2)}\right) \neq 0,~\text{for all } \sigma \in S_m.
\]
We choose the matrix $u_{i_0}$ as follows. For $l \leq m-1$ and $1\leq s\leq m-l$, choose the entries to be
\begin{align*}
    u_{i_0}(1,2) & = u_{i_0}(2,3) = \cdots = u_{i_0}(l-1,l) = 1\\
    u_{i_0}(l,l+1) & = \left( \lambda_{l-1,l+1}\cdot u_{i_0}({l-1,l}) \right) ^{-1} \left[ x_{l-1,l+1} \right] 
    \\
    u_{i_0}(l,l+2) & = \left( \lambda _{l-1,l+2} \cdot u_{i_0}({l-1,l}) \right) ^{-1} \left[ x_{l-1,l+2} - \beta _{l-1,l+2} \right] 
    \\
    \vdots & \kern 2cm \vdots \kern 2cm \vdots \kern 4cm \vdots
    \\
    u_{i_0}(l,l+s) & = \left(\lambda _{l-1,l+s}\cdot u_{i_0}({l-1,l})\right) ^{-1} .\left[ x_{l-1,l+s} - \beta _{l-1,l+s} \right],
\end{align*} 
for suitable choices of $\beta$'s. Further take $u_{j_0}$ to be the matrix with $(i,i)$ diagonal entries to be $\vec{a}_{ii}(j_0)$ and all others above the diagonal entries to be $1$. For all $i\neq i_0,j_0$ we take \[u_i=\diag\left(\Vec{a}_{11}(i),\Vec{a}_{22}(i),\ldots,\Vec{a}_{mm}(i)\right).\] Then with the prescribed choices, we have that $p\left(u_1,u_2,\ldots,u_n\right)=X$. Now the result is immediate.
\end{proof}
\color{black}
\hf Accumulating \autoref{lem:ZariskiDense}, \autoref{thm:order_m-1}, \autoref{thm:order_1}, \autoref{thm:order_t}, \autoref{eg:counterExample}, \autoref{prop:Waring_all_i_same} and \autoref{prop:Waring-ord=t}, we now state the following theorem which follows from the mentioned results.

\begin{thm}\label{thm:mainThm}
    Let $K$ be an algebraically closed field, $p\in K\left[ F_n \right] $ be a non-commutative polynomial such that $\ord=t$ such that the following results hold: 
    \begin{enumerate}
        \item If $t=0$, then $p\left( T_m(K) \right) $ is Zariski dense in $T_m(K)$.
        \item If $t=1$, then $p\left( T_m(K) \right) = T_m(K)^{(0)}$.
        \item If $1<t<m-1$, then $p\left( T_m(K) \right) \subset T_m(K)^{(t-1)}$, and equality might not hold in general. Furthermore, for every $m$ and $t$ there exists $d$ such that each element of $T_m(K)^{(t-1)}$  can be written as $d$ many elements from $p\left( T_m(K) \right) $. 
        \item If $t=m-1$, then $p\left( T_m(K) \right) = T_m(K)^{(m-2)}$.
        \item For $t\geq m$, $p\left( T_m(K) \right) = \{0\}$.  
    \end{enumerate}
\end{thm}

\hf We believe that it is true that for all $m\geq 5$ and $2\leq \ord\leq m-2$,  the value of $d$ is $2$ and hence we propose the following conjecture.
\begin{conjecture}
    Let $p\in K[F_n]$ be a polynomial in non-commutative variables. Suppose $\ord = t$ where $1< t< m-1$. Then $p \left(\tk[m]\right) +  p \left(\tk[m]\right) = T_m(K)^{(t-1)}$.
\end{conjecture}
\section{Image of word maps on \texorpdfstring{$\tk[m]^\times$}{TmK}}\label{sec:groupMap}

Let $G$ be an algebraic group. For $k\ge 2$ the map $P_k:G\to G$ defined as $P_k(g)=g^k$ is known as the \textit{$k^{\text{th}}$-power map}. This section is concerned with the question when such a map has a dense image. A well known result of A. Borel states that the image of a word map (in particular a power map) on a semisimple algebraic group $G$ is Zariski dense, see  \cite{Bo83}. Note that the image of the power map is not always dense for real points of an algebraic group in real topology. For example, in case of $SL_2(\mathbb{R})$ the image of $P_2$ is not dense in real topology, but its image is Zariski dense in $SL_2(\mathbb{R})$. We present similar kind of result concerning Zariski density of the image of a word map on Borel subgroups of $GL_n(K)$ which are by definition maximal connected solvable subgroup. We will present a general result concerning any element $w\in F_n$.

\begin{thm}\label{thm:groupMap}
    Let $w\in F_n$ be a word which is not a group law. Then $w \left(\tk[m]^\times\right)$ is dense in $\tk[m]^\times$. Furthermore, for $B$ a Borel subgroup of $GL_{m}(K)$  we have that $w \left(B\right)$ is Zariski dense in $B$.
\end{thm}

\noindent The proof is similar to the proof of \autoref{lem:ZariskiDense}. For the completeness we are writing the proof.

\begin{proof}[Proof of \autoref{thm:groupMap}]
    Note that $T_m(K)^\times$ can be viewed as $\mathcal{K}=\left(K^\times\right)^{m}\times K^{\frac{m(m-1)}{2}}$ with the subspace topology.
    By \autoref{lem:diagonal}, we have that
    \begin{align*}
        \D=\left\{A=(a_{\alpha\beta})\in T_m(K)^\times:a_{\delta\delta}\neq a_{\sigma\sigma}\text{ for all }1\leq\delta\neq\sigma\leq m\right\}\subseteq w\left(T_m(K)^\times\right).
    \end{align*}
    Hence, it is enough to prove that $\D$ is dense in $T_m(K)^\times$. 

    Take a point $P=\left(a_{11}, a_{22}, \ldots, a_{mm},a_{12}, a_{23},\ldots,a_{m-1m},\ldots, a_{1m}\right)\in \mathcal{K}$ and an open set $U$ containing $P$. Hence, there exists an ideal $I$ of $K\left[y_1,\ldots,y_{\frac{m(m+1)}{2} }\right]$ such that $U=\left(K^{\frac{m(m+1)}{2}}\setminus Z(I)\right) \cap \mathcal{K}= \mathcal{K}\setminus Z(I)$, where $Z(I)$ denotes the zero set of $I$. This implies there exists $f\in I$ with $f(P)\neq 0$ that is
    \begin{displaymath}
        f \left(a_{11}, a_{22}, \ldots, a_{mm},a_{12}, a_{23},\ldots,a_{m-1m},\ldots, a_{1m}\right)\neq 0.
    \end{displaymath} 
    If $a_{ii} \neq a_{j j}$ for all $i\neq j$, then $P\in \mathcal{D}$. So we will assume that there exists some $i_0\neq j_0$ but $a_{i_0i_0}=a_{j_0j_0}$. Without loss of generality assume that $i_0=1$ and $j_0=2$. Set $a_{11}' = a_{11}$. We want to find $a_{22}'\notin \left\{a_{22},0\right\}$ such that $f \left(a_{11}', a_{22}', \ldots, a_{mm},a_{12}, a_{23},\ldots,a_{m-1m},\ldots, a_{1m}\right)\neq 0$. Define 
    \begin{displaymath}
        h_{22}(x) = f \left(a_{11}', x, \ldots, a_{mm},a_{12}, a_{23},\ldots,a_{m-1m},\ldots, a_{1m}\right).
    \end{displaymath} 
    Note that the function is not identically zero. As $K$ is an algebraically closed field (hence infinite), there exists $a_{22}'\notin \left\{a_{11}',0\right\}$  such that $h_{22}\left(a_{22}'\right)\neq 0$. 

    Since $K$ is infinite, this procedure can be continued to obtain an element 
    \begin{displaymath}
        P'=\left(a_{11}',a_{22}',\ldots,a_{mm}',a_{12},\ldots,a_{1m}\right)
    \end{displaymath} 
    such that $a_{ii}'\neq a_{j j}'\neq 0$ for all $i\neq j$  and $f \left(P'\right)\neq 0$. Hence, $P'\in \mathcal{D}$. 

    The later part of the theorem follows from the fact that all Borel subgroups of $GL_m(K)$ are conjugate.
\end{proof}

\noindent\textbf{Acknowledgements:} We are immensely thankful to the reviewers for reading the paper very carefully and pointing out the flaws therein (in particular, about the choices of matrices in \autoref{thm:mainThmFor_n=4} and \autoref{thm:order_t}). We are extremely thankful to the editors of the Journal of Algebra, in particular, Prof. Gunter Malle and Prof. Louis Rowen for their help during the communication of this paper.

The first author (Panja) acknowledges the support of NBHM PhD fellowship. The second author (Prasad) was supported by UGC (NET)-JRF fellowship. We would also like to thank Dr. Aritra Bhowmick from IISER Kolkata and Dr. Jyoti Dasgupta from IISER Pune for their helpful discussion. Thanks are due to the Department of Mathematics and Statistics IISER Kolkata for organizing the workshop IINMM-2022, as we started working on this problem during this workshop. We are thankful to Dr. Somnath basu (IISER Kolkata), Prof. Anupam Kumar Singh (IISER Pune), and Prof. Manoj Kumar Yadav (HRI, Pryagraj), for their valuable suggestions.

\bibliographystyle{siam}

\end{document}